\documentclass[11pt]{article}

\usepackage{amssymb,amsfonts,amsmath,amsthm,setspace}
\usepackage[utf8]{inputenc}
\usepackage{fullpage}

\newtheorem{theo}{Theorem}
\newtheorem{prop}[theo]{Proposition}
\newtheorem{exam}[theo]{Example}
\newtheorem{lemm}[theo]{Lemma}

\theoremstyle{remark}
\newtheorem{rema}[theo]{Remark}
\theoremstyle{definition}
\newtheorem{defi}[theo]{Definition}
\newtheorem*{claim}{Claim}

\newtheorem{manualtheoreminner}{Theorem}
\newenvironment{manualtheorem}[1]{%
	\renewcommand\themanualtheoreminner{#1}%
	\manualtheoreminner
}{\endmanualtheoreminner}

\newcommand{\R}{\mathbb R}
\newcommand{\Pb}{\mathcal P}
\newcommand{\N}{\mathbb N}

\newcommand{\esup}{\mbox{\it essup\ }}

\newcommand{\varX}{\mathfrak X}
\newcommand{\varY}{\mathfrak Y}
\newcommand{\xbar}{\overline{x}}
\newcommand{\zero}{\mathcal N}
\newcommand{\B}{\mathcal B}

\def\v{\varphi}

\usepackage{xcolor}

\def \sp#1{\langle #1\rangle}

\def\ICM{infinitely cyclically monotone }

\title{60 years  of cyclic monotonicity: a survey}
\author{L. De Pascale, A. Kausamo, K. Wyczesany}
\date{}

\begin{document}

\maketitle

\begin{center}
    \thispagestyle{empty}
    \vspace*{\fill}
 \textit{Dedicated to Prof. R. T. Rockafellar in occasion of his 90th birthday.}
    \vspace*{\fill}
\end{center}

\section{Aims and scopes}
The primary purpose of this note is to provide an instructional summary of the state of the art regarding cyclic monotonicity and related notions. We will also present how these notions are tied to optimality in the optimal transport (or Monge-Kantorovich) problem, even though cyclic monotonicity is simply a property of a set and can be studied on its own merit.

Cyclic monotonicity arose from the problem of characterising the subgradient of a convex function. In \cite{minty}, Minty proved monotonicity, which can also be called 2-cyclic monotonicity, of the subgradient. To the best of our knowledge, the definition of cyclic monotonicity was first introduced in 1966 by Rockafellar in \cite{rockafellar1966characterization}, where he proved that maximal cyclic monotonicity characterizes the subgradients of convex functions. 

At the beginning of the '90s, the study of cyclic monotonicity met the Monge-Kantorovich mass transport problem, and it is the connection between the two problems that brought out numerous developments and generalisations as well as many new problems. We will highlight these links and regard cyclic monotonicity as a geometric notion of optimality of a transport plan (see Section \ref{sec:cyclic_mon} for a review of the relevant notions).  

The optimal transport or Monge-Kantorovich problem is widely known in the literature, and there are several extensive books that allow readers of any level to learn the basic as well as the advanced theory (e.g. see \cite{ambrosio-gigli,villani-book, santambrogio2015optimal}). Let $X$ and $Y$ be topological spaces and let $\mu$ and $\nu$ be Borel probability measures on $X$ and $Y$, respectively. We denote by $p_Z$ the projection onto a subspace $Z$, and we set
\[
\Pi (\mu, \nu) =\{ \gamma \in \Pb (X \times Y) \ : \ p_X (\gamma)=\mu, \ \  p_Y (\gamma) =\nu\}.
\]
This is the set of all {\it transport plans} between $\mu$ and $\nu$, i.e. all probability measures $\gamma$ on the product space $X\times Y$ such that for any Borel sets $A\subset X, \, B\subset Y$ we have $\gamma (A\times Y) = \mu (A)$ and $\gamma (X\times B)= \nu (B)$.
Given a Borel cost function $c: X \times Y \to \R \cup \{+ \infty\}$ the Kantorovich formulation of the optimal transportation problem is to find a plan that minimises the total cost, that is 
\begin{equation}\label{kantorovich}
	\min_{\gamma \in \Pi (\mu,\nu)} \int c(x,y) d \gamma.
\end{equation}

It is intuitively clear that if a plan $\gamma$ is optimal with respect to a cost $c$, then for any number $m\in \N$ of points $\{(x_i,y_i)\}_{i=1}^m \in {\rm supp}(\gamma)$, the total transport cost of this coupling, which can be written as $\sum _{i=1}^m c(x_i,y_i)$, should increase if instead we pair $x_i$'s and $y_j$'s in a different way. This is exactly the concept behind \emph{$c$-cyclic monotonicity} that will be studied in detail in the following section.

\section{Cyclic monotonicity, potentials and generalizations}\label{sec:cyclic_mon}

\subsection{Monotonicity and cyclic monotonicity}

Let $X=Y=\R^n$. Recall that a function $f:X\to Y$ is monotone if for any $(x_i,y_i), (x_j,y_j) \in ~{\rm graph} (f)$ we have that 
\[\sp{x_i-x_j, y_i-y_j} \ge 0.\]

This classical notion, also referred to as $2$-monotonicity, can be generalised to $N$-point-cyclic monotonicity, and further to cyclic monotonicity, which we define below, and which was  introduced by Rockafellar in \cite{rockafellar1966characterization}. Let us remark that Rockafellar considered a more general setting, where $X$ is a topological vector space over $\R$ and $Y=X^*$ is its topological dual.

By ${\mathcal S_k}$ we denote the set of all permutations on $k$ elements. 

\begin{defi} A subset $\Gamma \subset X \times Y$ is \emph{cyclically monotone} if for all $k \in \N$ and for all $\{(x_i, y_i)\}_{i=1}^{k} \subset \Gamma$ 
\[   \sum_{i=1}^k \langle  x_i-x_{\sigma(i)},  y_i \rangle \geq 0,
\]
holds for any permutation $\sigma\in {\mathcal S_k}$, where we set $x_{k+1}:=x_1$.
\end{defi}  

By a basic theorem on permutations of $k$ objects the above condition is equivalent to simply requiring that $\sum_{i=1}^k \langle  x_i-x_{i+1},  y_i \rangle \geq 0$ holds for all $k\in \N$ and any $k$-tuple of points $\{(x_i, y_i)\}_{i=1}^{k} \subset \Gamma$.
Moreover, note that cyclic monotonicity is a property of a set, and that if $\Gamma$ is cyclically monotone then any subset of $\Gamma$ has the same property.
An important fact is that if $\v : X \to [-\infty,+\infty]$ is a convex function then $\partial \v = \{(x,y): \, \v(z)-\v(x)\ge \sp{y,z-x}, \ \forall z \}\subset X\times Y$, is cyclically monotone.
$\partial \v$ (see, e.g. the books \cite{moreau1966functional,rockafellar1997convex,rockafellar2009variational}) is usually called the ``subdifferential''  of $\v$, and any $y \in \partial \v (x)$  is a subgradient of $\v$ at $x$.

The celebrated Rockafellar's theorem states that the inverse is also true: 
\begin{theo}[{\cite[Theorem 1]{rockafellar1966characterization}}, see also {\cite[Theorem 24.8]{rockafellar1997convex,rockafellar2015convex}}]\label{rockafellar} Let $\, \Gamma \subset X \times Y$ be a cyclically monotone set. Then there exists a convex function 
$\v: X \to \R$ such that $\Gamma \subset \partial \v$. 
\end{theo}

We say that a function $\v$ is a \emph{potential} for $\Gamma$ if its subgradient contains the set $\Gamma$.
The proof of this theorem is constructive and we present it in the more general setting of Theorem \ref{RocRuc}.

Rockafellar's result may be  further refined proving that  $\Gamma = \partial \v$ (so it is not a proper subset)  if it is maximal ({\it i.e.}  not properly contained in any cyclically monotone set).  Maximal $2$-cyclic monotone sets were studied, in particular, by Minty \cite{minty}.

It turns out that cyclic monotonicity can be equivalently described using the square of the Euclidean norm, denoted by $\| \cdot\|$. Indeed,  a set $\Gamma \subset X \times Y$ is cyclically monotone if  for all $k \in \N$ and for all $\{(x_i, y_i)\}_{i=1}^{k} \subset \Gamma$ 
\[
\sum_{i=1}^k \| x_i-y_i \|^2 \leq \sum_{i=1}^k \| x_i-y_{i+1} \|^2. 
\]
Clearly, this can be seen by expanding the above expression (see also Remark \ref{rem:non-mixing}). Further, this observation is also true if $X$ is a Hilbert space $H$ and $Y=X^*$ with the canonical identification $X^*=H$, where we denote the norm in $H$ by $\| \cdot \| $.

\subsection{Extensions of cyclic monotonicity}
Some natural extensions of the notion above will be of interest. The first one involves a more general cost function.

\subsubsection{$c$-cyclic monotonicity}
Let $X $ and $Y$ be two sets and let
\[c: X \times Y \to \R \cup \{ + \infty \}\] be a \emph{cost function}. 
We say that $\Gamma \subset X \times Y$ is \emph{$c$-cyclically monotone} if for all $k \in \N$ and $~\{(x_i, y_i)\}_{i=1}^{k} \subset \Gamma$  and for any permutation $\sigma \in {\mathcal S}_k$   
\[
\sum_{i=1}^k  c(x_i,y_i ) \leq \sum_{i=1}^k  c(x_i,y_{\sigma(i)}).
\]

Clearly, cyclic monotonicity corresponds to the choice  $c(x,y)=-\sp{x,y}$, or equivalently $c(x,y)=\|x-y\|^2/2$, which we will refer to as the \emph{quadratic cost}.

\begin{rema}\label{rem:non-mixing}
Note that, as in the case of the quadratic cost, adding 'non-mixing' terms to the cost does not influence $c$-cyclic monotonicity. More precisely, let $c(x,y)$ be a cost function and define $\tilde{c}(x,y)=c(x,y)+p(x)+q(y)$ for some functions $p,q$. Then a set is $c$-cyclically monotone if and only if it is $\tilde{c}$-cyclically monotone (see e.g. \cite{smith1992hoeffding}).
\end{rema} 

It was shown in \cite{rochet,ruschendorf} that if the cost $c$ attains only finite values then, as in the case of cyclic monotonicity, a set is $c$-cyclically monotone if and only if it is contained in the graph of the $c$-subdifferential of a $c$-class function. Let us recall the relevant definitions.  

\begin{defi}
 The \emph{$c$-transform} of a function $\v:X\to [-\infty, +\infty )$ is a function $\v^c: Y\to [-\infty, +\infty]$ given by 
\[ \v^c(y) = \inf_{x\in X} \left(c(x,y)-\v(x)\right).\]   
Similarly, when $\psi: Y\to [-\infty, +\infty)$, we use the same notation to define
\[ \psi^c(x) = \inf_{y\in Y} \left(c(x,y)-\psi(y)\right).\]
\end{defi}

This type of conjugate functions appeared in \cite{moreau1966functional} and is an elemental notion in the theory of optimal transport.
The $c$-transform is order-reversing (with respect to a pointwise order) and it is an involution on its ``$X$-image'' and ``$Y$-image'', respectively. In many applications $X=Y$ and $c$ is symmetric, i.e. $c(x,y)=c(y,x)$, in which case the image 
\[{\mathcal C} := \{\v:X\to [-\infty, +\infty )\ : \, \exists \psi:X\to [-\infty, +\infty ) \text{ such that } \psi^c=\v \}
\] 
is called the \emph{$c$-class}.

\begin{defi}\label{cheguaio}
Given a function $\varphi: X\to [-\infty,+\infty)$ its \emph{$c$-subgradient} is given by
\[\partial^c\varphi=\{(x,y)\in X\times Y~:~c(x,y)<\infty\text{ and }c(x,y)-\varphi(x)\le c(z,y)-\varphi(z)\text{ for all }z\in X\}.\]
\end{defi}
\begin{rema}
There is an alternative way of defining the $c$-subgradient of a function $~{\varphi: X\to [-\infty,+\infty)}$, that is
\[\tilde\partial^c\varphi=\{(x,y)\in X\times Y ~:~\varphi(x)+\varphi^c(y)=c(x,y)<+\infty\}.\]
Since it is possible that $\varphi(x)=-\infty$ and $\varphi^c(y)=+\infty$, some assumptions must be added to ensure that this definition is well-given. For example, if we know that 
\[\text{There exists }x'\in X\text{ such that }c(x',y)\in\R\text{ and }\varphi(x')\in\R,\]
then $\varphi^c(y)\in \R\cup\{-\infty\}$ and $(x,y)\in \partial^c\varphi$ if and only if $(x,y)\in \tilde\partial^c\varphi$. 
\end{rema}

Clearly, for a symmetric cost,  $(x,y) \in \partial^c \v$ if and only if $(y,x) \in \partial^c \v^c$. 

The following theorem is the generalisation of Rockafellar's theorem due to Rochet and R\"uschendorf.

\begin{theo}[{\cite[Theorem 1]{rochet}}, {\cite[Lemma 2.1]{ruschendorf}}]\label{RocRuc}
Let $X,\, Y$ be any sets and let $c:X \times Y\to \R$ be a cost function. A set $\Gamma\subset X\times Y$ is  $c$-cyclically monotone if and only if there exists a $c$-class function $\v$ such that $\Gamma \subset \partial^c \v$, i.e. $\v$ is a potential for $\Gamma$.
\end{theo}
\begin{proof}
Let $\Gamma$ be a $c$-cyclically monotone set. Let $(x_0,y_0) \in \Gamma$ and define $\varphi: X \to \R$ as follows
\begin{equation}\label{potential}
\begin{split}
	\varphi(x)= \inf \{c(x,y_n)-c(x_n,y_n)+c(x_n,y_{n-1})-c(x_{n-1},y_{n-1})+\dots+\\ c(x_1,y_0)-c(x_0,y_0):~  
	(x_i,y_i)\in \Gamma \ \mbox{for}\ i=1,\dots,n, \ n\in \N.
	\}  
 \end{split}
 \end{equation}
Since $\Gamma$ is $c$-cyclically monotone, $\varphi (x_0)=0$  (taking $n=1$ and $(x_1,y_1)=(x_0,y_0)$) and $\varphi$ is a $c$-class function. 
To show that $\Gamma \subset \partial^c \varphi$ consider $(x',y')\in \Gamma$, and $\varepsilon>0$ and let $(x_1,y_1), \dots, (x_n,y_n)\in \Gamma$ be a chain such that 
\[ c(x',y_n)-c(x_n,y_n)+c(x_n,y_{n-1})-c(x_{n-1},y_{n-1})+\dots+ c(x_1,y_0)-c(x_0,y_0) \leq  \varphi (x')+\varepsilon.\]
Let $z\in X$, since $(x',y') \in \Gamma$  we can consider the chain $(x_1,y_1), \dots, (x_n,y_n),(x',y')$ among the admissible chain for the definition of $\varphi(z)$ and get 
\[\varphi(z)  \leq c(z,y') -c(x',y') +\varphi (x')+\varepsilon.\]
Since $\varepsilon$ is arbitrary this implies $(x',y')\in \partial^c \varphi$. Since $\varphi(x_0)=0$ and $c$ assumes only real values, this last inequality, choosing $(x',y')=(x_0,y_0)$, also implies $\varphi (z) < +\infty$.
\end{proof}

The above theorem guarantees the existence of a potential for finite-valued cost functions. However, if the cost does attain the value $+\infty$ this statement is no longer true. See, for instance, \cite{artstein2022rockafellar} where the authors present  an example of a set that is $c$-cyclically monotone but whose graph is not supported on  the $c$-subgradient of any function. 
In the attempt to use formula \eqref{potential} above in a more general setting, in \cite{Schachermayer-optimal-and-better}, the notion of \emph{$c$-connectivity} is proposed and is shown to be sufficient to guarantee the existence of a potential.
\begin{defi}\label{connecting}
Let $X,Y$ be sets, $c:X\times Y\to \R\cup \{+\infty\}$ be a cost function and $\Gamma\subset X\times Y$. We denote $(x,y)\precsim (\tilde x,\tilde y)$ if there exist pairs $(x_0,y_0), \dots, (x_n,y_n)\in\Gamma$ with $(x_0,y_0)=(x,y)$ and $(x_n,y_n)= (\tilde x, \tilde y)$ such that $c(x_1,y_0), \dots, c(x_n,y_{n-1})< +\infty$.
If $(x,y)\precsim(\tilde x,\tilde y)$ and $(\tilde x, \tilde y)\precsim (x,y)$, we denote $(x,y)\approx (\tilde x,\tilde y)$. We say that $\Gamma$ is \emph{$c$-connecting} if $c(x,y)$ is finite for all $(x,y)\in\Gamma$ and if for all $(x,y),(\tilde x,\tilde y)\in\Gamma$ we have $(x,y)\approx (\tilde x,\tilde y)$. 
\end{defi}
Note that $\approx$ is an equivalence relation on $X\times Y$. 
If the cost function is real-valued then all sets are trivially connecting. Hence, $c$-connecting sets do not need to be $c$-cyclically monotone and so, on its own, $c$-connectivity does not capture optimality in a way that $c$-cyclic monotonicity does.
However, if a $c$-connecting set $\Gamma$ is also $c$-cyclically monotone then, for any cost function, there exists a $c$-potential which can be obtained using Rockafellar's approach. Indeed, the $c$-connectivity, together with $c$-cyclical monotonicity, implies that the function $\varphi$ defined by the formula \eqref{potential} is real-valued everywhere on the projection $p_X (\Gamma)$, which means that one can proceed with the proof as in the case of a real-valued cost function $c$. 

However, in general, the $c$-connectivity of $\Gamma$ is not a necessary condition for the existence of a potential, which can be seen on simple discrete examples.
In \cite{artstein2022rockafellar}, the authors show that for general costs \emph{$c$-path-boundedness} is the notion equivalent to the existence of a potential. We discuss it in the next subsection.

\subsubsection{$c$-path-boundedness}

When the cost function is allowed to take infinite values, an important case that has been considered on the case-by-case basis (e.g. see \cite{bertrand-gauss-curvature,bertrand-pratelli-puel,bertrand-puel,brenier-relativ-cost,gangbo-oliker}), the $c$-cyclic monotonicity is not sufficient to guarantee the existence of a potential. It was shown in \cite{artstein2022rockafellar} that a natural condition to consider is $c$-path-boundedness.

\begin{defi}\label{def:c-path-bdd}
Fix sets $X,\,Y$ and $c:X\times Y\to \R\cup \{+\infty\}$. We say that $\Gamma\subset X\times Y$  is \emph{$c$-path-bounded} if $c(x,y)<\infty$ for any $(x,y)\in \Gamma$, and for any $(x,y)\in \Gamma$ and $(z,w)\in \Gamma$, there exists a constant $M=M((x,y), (z,w))\in \R$ such that for any $m\in {\mathbb N}$ and any $\{(x_i, y_i): 1\le i\le m\} \subset \Gamma$ we have 
\[ c(x,y)-c(x_1,y) + \sum_{i=1}^{m-1}\big(c(x_i, y_i)-c(x_{i+1}, y_{i})\big) + c(x_m,y_m)-c(z,y_m) \le M \]
\end{defi} 

A $c$-path-bounded set must be $c$-cyclically monotone; indeed, if there is some cycle for which the sum is positive, then, choosing one point on this cycle to be both the beginning and the end of a path, one can replicate it many times to get paths with arbitrarily large sums.  Further, if the cost is finite-valued then $c$-path-boundedness and $c$-cyclic monotonicity are equivalent as we may choose $M=c(x,w)-c(z,w) <+\infty$.

As mentioned before, for costs attaining infinite values, $c$-connectivity together with $c$-cyclic monotonicity implies $c$-path-boundedness.
However, a $c$-path bounded set is not always $c$-connecting. Indeed, consider, for example, any cost function $c:\R\times \R \to \R \cup \{+\infty\}$, which is finite on the set $S:= \{(x,y)\in \R^2:\ x+y> 0\}$
 and $+\infty$ otherwise. We may easily choose $\Gamma$ to consist of two points $(x_1,y_1),\, (x_2,y_2)$ such that $(x_i,y_i)\in S$ for $i=1,2$, and which form a $c$-cyclically monotone set, but such that $(x_1,y_2)\notin S$ which means that $\Gamma$ is not $c$-connecting. But since there are only two points in $\Gamma$ it is easy to check that it is $c$-path-bounded, and one can easily find a potential. For a much more elaborate example but with more details provided see Example \ref{ex:connecing+c-pbdd} in Appendix \ref{appendix:example}.

The following theorem established in \cite{artstein2022rockafellar} states that $c$-path-boundedness of a set is equivalent to the existence of a potential. 

\begin{theo}[{\cite[Theorem 1]{artstein2022rockafellar}}]
	Let $X,\,Y$ be two arbitrary sets and $c:X\times Y\to  \R\cup \{+\infty\}$ an arbitrary cost function. For a given subset $\Gamma \subset X\times Y$ there exists a $c$-class function $~{\varphi:X\to[-\infty, +\infty]}$ such that $\Gamma\subset \partial^c \varphi$ if and only if	$\Gamma$ is  $c$-path-bounded. 	
\end{theo}

The proof of this theorem differs greatly from the proofs in the case of finite-valued costs. The authors reformulated the problem of finding a $c$-potential for a given set as the problem of existence of a solution to a special family of linear inequalities. More precisely, the following theorem was proved in \cite{artstein2022rockafellar}.

\begin{theo}[{\cite[Theorem 3.1]{artstein2022rockafellar}}]\label{thm:potential-means-inequalities}
	Let $c: X \times Y \to    \R\cup \{+\infty\}$  be a cost function and let $\Gamma\subset X\times Y$. 
	Then there exists a $c$-potential for $\Gamma$, namely a $c$-class function $~\varphi: X\to [-\infty, +\infty]$   such that $\Gamma \subset  \partial^c\varphi $  if and only if the following system of  {inequalities}, 
	\begin{equation}\label{eq:system-of-eq-for-potential}
	c(x,y) - c(z,y) \le \varphi(x) - \varphi(z), \end{equation}
	indexed by $(x,y), (z,w) \in \Gamma$, has a solution $\varphi: p_X \Gamma \to \R$,  
	where $~{p_X\Gamma = \{ x  \in X: \exists y   \in Y , (x,y)\in \Gamma\}}$.   
\end{theo}


\begin{rema}{\bf  The inequalities in the above theorem can be equivalently indexed by pairs $((x,y),z)\in \Gamma\times p_X\Gamma$, where $w$ is ignored as it does not appear in the inequalities. Moreover, note that the solution vector $(\v(x))_{x\in p_X\Gamma}$ is indexed by $p_X\Gamma$. One may worry that in fact we obtain a vector $(\v(x,y))_{(x,y)\in G}$, which seems multi-valued. However, the authors observe that if $(x,y)$ and $(x,y')$ are both in $\Gamma$ then, writing out the relevant inequalities for these two points, we get that $\v(x,y)=\v(x,y')$. In other words, the solution vector indeed only depends on the first coordinate.} \end{rema}

Intuitively, the above theorem says the following: since for every $(x_i,y_i)\in \Gamma$ letting $\varphi_{i}(x)= c(x,y_i)-a_i$ we have that $\partial ^c \varphi_{i} (x_i) =y_i$ for any $a_i\in \R$, what one needs to take care of is the ``good gluing'' of such functions, i.e. the existence of constants $a_{i}$ such that the function ${\varphi_a (x) := \min_{(x_i,y_i)\in \Gamma} \varphi_{i} (x)}$ satisfies for every $x_i\in p_X \Gamma$ that $\varphi_a(x_i) = \varphi_i (x_i)$. Therefore, finding a potential amounts to showing that there exists a solution $a$ to the family of inequalities
\[c(x_i, y_i) - a_i \le c(x_i, y_j) - a_j.  \] 

The authors then show that given $\{\eta_{i,j}\}_{i,j \in I}\in [-\infty, +\infty)$, where $I$ is some arbitrary index set and with $\eta_{i,i}=0$, the system of inequalities 
	\[\eta_{i,j} \le a_i - a_j, \quad i,j \in I \]	has a solution if and only if for any $i,j\in I$ there exists some constant $M(i,j)$ such that for any $m$ and any $i_2,\cdots,i_{m-1}$, letting $i = i_1$ and $j = i_m$ one has that  $\sum_{k=1}^{m-1} \eta_{i_k, i_{k+1}}\le M(i,j)$. 	 This directly translates into $c$-path-boundedness. 
	
The proof of the solvability of the system of inequalities uses Zorn's lemma. The authors also give a new proof of Theorem \ref{RocRuc}, which can be found in \cite{wyczesany_thesis}. \vspace{2mm}

Having established the existence of a potential one may want to consider its regularity. This is a deep and hugely rich theory, which despite being related, is different to what we consider in this note. Therefore, we will not address these issues. We refer the interested reader to the second book of C. Villani \cite{villani-book}, which offers a useful starting point and many further references. 

\subsubsection{Adding more variables}

Finally, one may consider a \emph{multi-dimensional} version of the same ideas, which to the best of our knowledge, was first introduced by Kim and Pass in \cite{kim2014general}. Let $X_1, \dots, X_N$ be sets and 
\[
c: X_1 \times \dots \times X_N \to \R \cup \{+ \infty\}.
\]
\begin{defi}\label{multicyclic}
We say that $\Gamma \subset X_1 \times \dots \times X_N$ is \emph{$c$-cyclically monotone} if for all $k \in \N$ and $\{(x^i_1,\dots, x^i_N)\}_{i=1}^{k} \subset \Gamma$ and for all permutations $ \sigma_2, \dots, \sigma_N \in {\mathcal S}_k$,  we have
\[
\sum_i c(x^i_1, x^i_2, \dots, x^i_N) \leq  \sum_i c(x^i_1, x^{\sigma_2 (i)}_2\dots, x^{\sigma_N(i)}_N). 
\]
\end{defi}

In correspondence to the $c$-subgradient for two variables, for the multivariable case one considers $c$-splitting sets. 

\begin{defi}\label{cm}
A set $\Gamma \subset X_1\times \ldots \times X_N$ is called \emph{$c$-splitting} if there exist $N$ functions $\varphi_i: X_i\to [-\infty,+\infty)$ such that 
\[\varphi_1 (x_1)+\varphi_2 (x_2)+\ldots +\varphi_N (x_N) \le c(x_1,x_2, \ldots,x_N)
\] holds for all $(x_1,x_2, \ldots ,x_N)\in X_1\times \ldots \times X_N$, and
\[\varphi_1 (x_1)+\varphi_2 (x_2)+\ldots +\varphi_N (x_N) =  c(x_1,x_2, \ldots,x_N)
\] 
holds for all $(x_1,x_2, \ldots ,x_N)\in \Gamma$. We call the functions $(\varphi_1,\ldots , \varphi_N)$ a $(\Gamma,c)$-splitting tuple.
\end{defi}

Similarly to the two variables case, it is not difficult to check that a $c$-splitting set is $c$-cyclically monotone. The reverse implication will be discussed in Section \ref{mot}. 

The existence of the $(\Gamma,c)$-splitting tuple for a given set $\Gamma$ has been investigated in some special cases. In particular, Bartz, Bauschke and Wang \cite[Theorem 2.7]{bartz2018class} showed that if the cost function $c$ can be written in the form
\[ c(x_1,\ldots , x_N) =\sum_{1\le i<j\le N} c_{i,j}(x_i,x_j),
\]
where $c_{i,j}:X_i\times X_j \to \R$, then if for each $1\le i<j \le N$ the projection $\Gamma_{i,j}$ of $\Gamma\subset X_1\times \ldots \times X_N$ onto $X_i\times X_j$ is $c_{i,j}$-cyclically monotone, then $\Gamma$ is $c$-cyclically monotone, and further, there exist a $c$-splitting tuple $(\varphi_1, \ldots , \varphi_N)$ given by 
\[ \varphi_i(x_i) = \sum _{i<j\le N} \psi_{i,j}(x_i) +\sum_{1\le j<i} \psi _{j,i}^{c_{j,i}}(x_i),\] 
where for every $1\le i<j\le N$ we have $\psi_{i,j}:X_i\to [-\infty, +\infty]$ is the $c_{i,j}$-potential for $\Gamma_{i,j}$.

\begin{rema}
This statement can be now generalized to allow $c_{i,j}:X_i\times X_j \to \R\cup\{+\infty\}$ if we additionally assume that $\Gamma_{i,j}$ is $c_{i,j}$-path-bounded since, as we have seen, this guarantees the existence of the functions $\psi_{i,j}$.
\end{rema}  

Clearly, the condition of $c_{i,j}$-cyclically monotone projections is stronger than $c$-cyclic monotonicity on the whole set $\Gamma$, and the authors present the relevant example in \cite{bartz2018class}. However, it turns out that if we additionally assume that $X_1=\ldots=X_N=\R$ are one-dimensional, then the two conditions are equivalent. 

Let us remark that the maximal multimarginal $c$-monotonicity for the ``quadratic'' multimarginal cost (in the spirit of Minty \cite{minty}) was considered in \cite{bartz2021multi}.

\section{Classical optimal transport}

In this section we collect results connecting $c$-cyclic monotonicity and its various extensions to optimality of a transport plan in the Monge-Kantorovich problem \eqref{kantorovich}. Prior to the detailed exploration of the deep relationship between $c$-cyclic monotonicity and the optimality of transport plans, let us recall the classical setup of two marginal transport problems and justify why in this case one can expect a minimiser of the transport problem to exist.

Let $X$ and $Y$ be topological spaces and let $\mu$ and $\nu$ be Borel probability measures on $X$ and $Y,$ respectively. Let
\[ 
c: X \times Y \to \R \cup \{+ \infty\}
\]
be a Borel cost function that may assume the value $+\infty$. We say that $c:X\times Y \to \R\cup \{+\infty\}$ is \emph{essentially bounded} with respect to $\mu$ and $\nu $ if there exist upper semi-continuous functions ${a:X\to (-\infty,+\infty]}$, $a\in L^1(\mu)$
and $b:Y\to (-\infty,+\infty]$, $b\in L^1(\nu)$ such that $c(x, y)\ge a(x) + b(y)$ for all $x\in X,\, y\in Y$. Since we may absorb the functions $a$ and $b$ into the cost, instead of writing that the cost is essentially bounded we will sometimes write that $c: X\times Y \to [0,+\infty]$. One needs this assumption to avoid the situation where there are plans with a total cost $-\infty$ (and hence nothing can be inferred about the optimal plan). 

If $c$ is assumed lower semi-continuous, then the `total cost' functional
\[
C[\gamma]:= \int c(x,y) d \gamma
\] 
is lower semi-continuous with respect to the weak$^*$ convergence of probability measures. This fact, together with Prokhorov's theorem \cite{prokhorov1956convergence} (see Th. 5.1.3 in \cite{ambrosio2005gradient}) that assures compactness of the set of all transport plans $\Pi (\mu,\nu)$, gives the existence of at least one minimiser of the transport problem \eqref{kantorovich}. This is captured in the following theorem which can be found, for example, in \cite{villani2021topics}.

\begin{theo}[{\cite[Theorem 4.1]{villani2021topics}}]\label{thm:villani} 
	Let $X,\, Y$ be two Polish spaces,  $\mu \in \mathcal{P} (X)$ and $\nu \in \mathcal{P} (Y)$. Let  $c:X\times Y \to \R\cup \{+\infty\}$ be a lower semi-continuous cost function which is essentially bounded with respect to $\mu  $ and $\nu$. Then there exists a $c$-optimal plan $\pi \in \Pi(\mu,\nu)$.
\end{theo}

In the above theorem,  the existence of a plan with finite total cost is not assumed (nor guaranteed) and when no finite cost plan exists,  any plan (say,  $\mu \otimes \nu$) is optimal.  Therefore, in the results to follow, it will be assumed that there exists a plan with finite total cost.  A transport plan will be called \textit{finite} if its total cost is finite. 

It is of interest to understand the structure of an optimal transport plan. It turns out that ``often" a transport plan is optimal if and only if it is concentrated on a $c$-cyclically monotone set. There is a vast array of this type of results, which we outline in this section.

\subsection{Optimality implies $c$-cyclic monotonicity}\label{sec:optimal-ccm}

Analysing the geometric structure of an optimal plan is a problem with a long history. It seems that $c$-cyclic monotonicity first appeared in this context in \cite{smith1992hoeffding} in the following equivalent formulation of the Kantorovich's problem
\begin{equation}\label{var-alea}
\min \{ \int c(\varX, \varY) d\omega\ : \  \varX\ \mbox{and}\  \varY\  \mbox{random variables with}\   Law (\varX)=\mu, \ Law (\varY)=\nu\}.
\end{equation}
The equivalence between problems \eqref{var-alea} and  \eqref{kantorovich} can be seen in the following way:
If $\varX$ and $\varY$ are admissible for \eqref{var-alea} then the joint law $\gamma= \varX\otimes \varY _\sharp \omega$ is a transport plan in $\Pi(\mu,\nu)$ and 
\[
 \int c(x,y) d \gamma=\int c(\varX, \varY) d\omega.
\]
On the other hand if $\gamma \in \Pi (\mu,\nu)$ then one can choose $(\Omega, \omega)= (X \times Y, \gamma)$ and 
$\varX= p_X$, $\varY= p_Y$.
In that context, which was partly motivated by some models appearing in financial mathematics, authors started by characterizing optimal random variables using $c$-cyclic monotonicity.
In such a model, $\varX$ and $\varY$ could represent the prices at a given moment of two assets or resources. In a robust financial framework the joint probability law of the portfolio $(\varX, \varY)$ cannot be determined from market information, however, the distribution of each price is known. 
The cost $c$ could be interpreted as an option whose pay-off is determined by the vector of prices $(\varX,\varY)$ whose distribution is then of interest (maximizing or minimizing is, clearly, the same).

 The classical structure of cyclic monotonicity of optimal plans was mentioned as a possible alternative tool in Brenier \cite{Brenier} and explicitly exploited in Caffarelli \cite{caffarelli1992regularity}. After that, Gangbo and McCann \cite{gangbo-mccann}	extended the result to lower semi-continuous cost functions bounded from below. They showed that every finite optimal plan with respect to such a cost lies on a $c$-cyclically monotone set.

For more general settings there are, essentially, two arguments to prove that the support of the optimal plan must be $c$-cyclically monotone. The first one uses duality and appeared in \cite{smith1992hoeffding}, while the second relies on modifying a transport plan that is not $c$-cyclically monotone and showing that its cost can be improved. The latter technique was introduced in \cite{abdellaoui1994distance} and used, for example, in Proposition 2.24 of Villani\footnote{At the time when the book was written the question of sufficiency of $c$-cyclic monotonicity was wide open. In fact, Open Problem 2.25 of the same book (later solved by Schechermeyer and Teichmann in \cite{SchachermayerTeichmann}) asks whether sufficiency holds for the quadratic cost.} \cite{villani2021topics} for the quadratic cost $\|x-y\|^2$  or in  Theorem 2.3 of \cite{gangbo-mccann} for a continuous, positive cost.

To the best of our knowledge, the most general result was proved by Beiglböck, Goldstern, Maresch, and Schachermayer \cite{Schachermayer-optimal-and-better} who removed regularity assumptions on the cost: \begin{theo}[{\cite[Theorem 1.a]{Schachermayer-optimal-and-better}}]\label{thm: optimal better 1A}
Let $X, \,Y$ be Polish spaces equipped with Borel probability measures $\mu, \nu$ and let $c : X \times Y \to [0, \infty]$ be a Borel measurable cost function. Then every optimal transport plan with finite total cost is $c$-cyclically monotone.
\end{theo}

The reverse implication ($c$-cyclic monotonicity implies optimality) is not true in general, and we collect known results and counterexamples in the rest of this section.

\subsection{When $c$-cyclic monotonicity implies optimality?}

We begin by presenting an example from \cite{ambrosio-pratelli}, which shows that one cannot always expect $c$-cyclic monotonicity of the support of the plan to imply optimality. In this example the cost is regular, namely lower semi-continuous. However, it assumes not only real values but also the value $+\infty$. 

\begin{exam}[{\cite[Example 3.1]{ambrosio-pratelli}}]\label{esempioovunque}
	Let $X=Y=[0,1]$ and let $\mu=\nu=\lambda$ be the Lebesgue measure. Letting $\alpha \in [0,1)$ be an irrational number define
	\[ \Gamma=\{(x,x):x\in X\}, \ \ \ \ \ \Gamma_\alpha =\{ (x,x\oplus \alpha): x\in X\},
	\]
	where $\oplus$ is the addition modulo $1$. Define the cost $c=a\in [0,\infty)$ on $\Gamma$, $c=b\in [0,\infty)$ on $\Gamma_\alpha$ and  $+\infty$ otherwise. 
	Now both sets $\Gamma$ and $\Gamma_\alpha$ are $c$-cyclically monotone. Define maps $T:X\to X\times Y$ to be $T(x)=(x,x)$
	and $T_\alpha :X\to X\times Y$ to be $T_\alpha(x)=(x,x\oplus \alpha)$. These maps induce transport plans $\pi=T_{\#} \lambda$ and $\pi_\alpha=T_{\alpha \,\#}\lambda$ which are supported on $\Gamma$ and $\Gamma_\alpha$,  respectively. Both of these plans have finite total cost, equal to  $a$ and $b$ respectively, which means that the optimality depends on the choice of values $a,b$, while, as already mentioned, both supports are $c$-cyclically monotone. 
\end{exam}

This example shows that one will have to assume something about the cost or the transported measures. It turns out that a joint property of measures and the cost is needed, which was shown by Beiglböck, Goldstern, Maresch, and Schachermayer in \cite{Schachermayer-optimal-and-better}.

\begin{theo}[{\cite[Theorem 1.b]{Schachermayer-optimal-and-better}}]\label{thm:closedsetnullset}
Let $X,\, Y$ be Polish spaces equipped with Borel probability measures $\mu, \, \nu$ and $c:X\times Y\to [0,\infty]$ a Borel measurable cost function. Assume that the set $\{(x,y)\in X\times Y:\, c(x,y)= +\infty\}$ is the union of a closed set $F$ and a $\mu\otimes\nu$-null set $N$. Let $\gamma\in\Pi(\mu,\nu)$ be a finite and $c$-cyclically monotone plan. Then $\gamma$ is optimal.
\end{theo}

This result is the state of the art, establishing when the equivalence of optimality and $c$-cyclic monotonicity holds. Nevertheless, in what follows, we are going to present some prior developments of this problem as we find it interesting and instructive. Moreover, it will give the reader the idea behind the proof of the above theorem.

\subsubsection{Finitely-supported measures} 
For finite spaces $X$ and $Y$ or, equivalently, for finitely-supported measures $\mu$ and $\nu$ one can prove that 
\begin{theo} \label{finitesets} Assume that $\mu$ and $\nu$ are finitely-supported Borel probability measures on Polish spaces $X$ and $Y$, respectively.  A transport plan $\gamma\in\Pi(\mu,\nu)$ is optimal with respect to the optimal transport problem corresponding to a cost function $c:X\times Y\to \R$ if and only if it is $c$-cyclically monotone. 
\end{theo}

The above theorem can be proved in several different ways. The first and, probably, more popular one, uses linear programming but, despite being part of the folklore, we cannot find a complete reference with proof. Several results for finitely supported measures are contained in \cite{peyre2019computational}. In the quadratic case, this result is left as an exercise (Exercise 2.21) in \cite{villani2021topics}. Further, we remark that it is a particular case of Theorem \ref{pratelliatomic} below. 
In Appendix \ref{appendix:finitely} we give another proof that uses a peculiar structure of  finitely-supported transport plans with the same marginals. More precisely, given two finitely-supported transport plans with the same marginals, it is possible to obtain one of them by permuting the coordinates of points in the support of the other.

Using Theorem \ref{finitesets} and the density of convex combinations of discrete measures in the space of probability  measures one can prove that there exists an optimal transport plan with $c$-cyclically monotone support for general marginal measures $\mu\in\mathcal{P}(X), \ \nu\in\mathcal{P}(Y)$ (see \cite{mccann1995existence} and Remark 2.26 in \cite{villani2021topics}).

\subsubsection{Purely atomic measures}
In \cite{pratelli2008sufficiency}, Pratelli shows that if the measures $\mu$ and $\nu$ are purely atomic (that is, concentrated on at most countable sets) then optimality and $c$-cyclic monotonicity are indeed equivalent.

\begin{theo}[{\cite[Theorem A]{pratelli2008sufficiency}}]\label{pratelliatomic}
Let $X$ and $Y$ be Polish spaces, $\mu\in\mathcal{P}(X)$ and $\nu\in\mathcal{P}(Y)$  purely atomic, and $c:X\times Y\to[0,\infty]$ a cost function. Then a finite transport plan $\gamma\in\Pi(\mu,\nu)$ is optimal if and only if it is concentrated on a $c$-cyclically monotone set. 
\end{theo}

In the same paper, using the above theorem and an approximation argument, Pratelli obtains a result for general measures under the assumption of continuity of the cost. 

\subsubsection{General measures}
As mentioned above, Pratelli \cite{pratelli2008sufficiency} shows, for any Borel measures, that if the cost is continuous the equivalence holds. 

\begin{theo}[{\cite[Theorem B]{pratelli2008sufficiency}}]\label{thm:Partelli cont cost}
Let $X$ and $Y$ be Polish spaces, $\mu\in\mathcal{P}(X)$ and $\nu\in\mathcal{P}(Y)$, and assume that the cost function $c:X\times Y\to[0,\infty]$ is continuous. Then a finite transport plan $\gamma\in\Pi(\mu,\nu)$ is optimal if and only if it is concentrated on a $c$-cyclically monotone set. 
\end{theo}

Instead of insisting on the cost being continuous, it turns out that the statement holds for any real-valued Borel cost. A posteriori, this can be seen from Theorem \ref{thm:closedsetnullset}. Indeed, for any finite-valued
cost function the set $\{(x,y): \, c(x,y)=+\infty\}$ is empty and hence satisfies the conditions of the theorem, which in turn means that optimality and $c$-cyclic monotonicity of the plan are equivalent. 
Since Theorem \ref{thm:closedsetnullset} has a complex proof we choose to
first show the statement for real-valued costs.  We use the fundamental steps from \cite{Schachermayer-optimal-and-better} and, for clarity, we add more detail. 

\begin{theo}\label{cFinite}
Let $X, \,Y$ be Polish spaces equipped with Borel probability measures $\mu, \nu$, and let $c : X \times Y \to [0, \infty)$ ($c$ is real-valued) be a Borel measurable cost function. Then every $c$-cyclically monotone transport plan $\gamma$ with finite total cost is optimal.	
\end{theo}

\begin{proof}  Let $\Gamma$ be a cyclically monotone set such that $\gamma(\Gamma)=1$. By Theorem \ref{RocRuc}, since $c$ is real-valued, there exists a function $\varphi:X\to\R\cup\{-\infty\}$ such that $\Gamma \subset \partial^c \varphi$.
In particular, this implies that 
\[c(x,y)-\varphi(x)\leq c(x',y)-\varphi(x'), \ \forall x'\in X, \ (x,y)\in \Gamma. \]	
In fact, Theorem \ref{RocRuc} gives a formula for $\varphi$, which we use to prove, in line with \cite{Schachermayer-optimal-and-better}, that $\varphi$ is universally measurable. This, together with an argument showing integrability, yields the optimality of the plan.

A crucial ingredient of the proof is the fact that the image of a Borel set via a Borel map is universally measurable (we recall the definition and the relevant facts about \emph{universal measurability} in Appendix \ref{sec:appendix universally}). 
To show this, we analyse the function defined in \eqref{potential} by  splitting the minimisation process in two steps. First we minimise  among chains of fixed length $n$, and define
\begin{multline*}
\varphi_n (x)= \inf \{c(x,y_n)-c(x_n,y_n)+c(x_n,y_{n-1})-c(x_{n-1},y_{n-1})+\dots+\\ c(x_1,y_0)-c(x_0,y_0):  \,
(x_i,y_i)\in \Gamma \ \ \mbox{for}\  \ i=1,\dots,n\}.  
\end{multline*}
Then we minimise with respect to $n$, which gives the potential
\[\varphi(x) = \inf_n \varphi_n (x).\]
Since $c$ is Borel measurable, we get that for each $n$ and any $n$-tuple of points in $\Gamma$ the function 
\[c(x,y_n)-c(x_n,y_n)+c(x_n,y_{n-1})-c(x_{n-1},y_{n-1})+\dots+\\ c(x_1,y_0)-c(x_0,y_0)\]
is a Borel measurable function on $X\times \Gamma^n$.
Hence, for each $n$ and $\alpha \in \R$ the set 
	\begin{align*}
	L_\alpha^n:= \big\{ (x,y_n,x_n, \ldots, y_0,x_0): \ c(x,y_n)-c(x_n,y_n)+&c(x_n,y_{n-1})-c(x_{n-1},y_{n-1})+ \\ &+\dots+ c(x_1,y_0)-c(x_0,y_0) <\alpha\big\}
	\end{align*}
	is a Borel subset of $X\times \Gamma^n$.
Since 
\[\{x : \ \varphi(x) < \alpha\} = \cup_n \{x : \ \varphi_n (x) < \alpha \}, \]
\[\{x : \ \varphi_n (x) < \alpha \} = p_X (L^n_\alpha) \]	
and $p_X$ is a Borel function, we obtain the universal measurability of $\varphi$. 

We now show that there exists a Borel measurable function $\tilde \varphi$ such that 
\[\tilde \varphi \leq \varphi \ \  \ \mbox{ and } \ \ \  \tilde \varphi= \varphi \ \  \mu-a.e.\]
Indeed let ${\mathcal I}=\{I_n\}$ be the family of half-closed intervals $[a_n, b_n)$ with rational endpoints. Observe that $\mathcal I$ generates the Borel $\sigma$-algebra on $\R$. 
Let $\tilde \mu$ be the completion of $\mu$ (for the definition see \eqref{def:completion} in Appendix \ref{sec:appendix universally}). Since $\varphi$ is universally measurable, $\varphi ^{-1}(I_n)$ is $\tilde \mu$-measurable and thus can be written as the union of sets $B_n \cup N_n$ with $B_n$ Borel measurable and $\mu(N_n)=0$. Let $N=\cup_n N_n$ we define 
\[\tilde \varphi (x) = \begin{cases}
-\infty & \mbox{if} \ x \in N,\\
\varphi(x) & \mbox{otherwise.}
\end{cases} \]   
We then consider  $\Gamma_1:=\Gamma \setminus p_X^{-1}(N)$ which is a $c$-cyclically monotone Borel set such that $\gamma(\Gamma_1)=1$ and 
\[c(x,y)-\tilde \varphi (x) \leq c(x',y)-\tilde \varphi (x'), \ \ \forall x' \in X, \ (x,y)\in \Gamma_1.    \]
Moreover $\tilde \varphi (x) \in \R$ on $p_X (\Gamma_1)$.
We proceed, as usual, passing to a $c$-transform of the candidate Kantorovich potential $\tilde \varphi$.
Define 
\[\psi (y) =\inf_{x\in p_X (\Gamma_1)} \{c(x,y)- \tilde \varphi(x)\}.\]
The function $\psi$ is universally measurable since 
\[\{y: \ \psi(y)<\alpha \} = p_Y (\{(x,y)\in \Gamma_1  :\ c(x,y)-\tilde \varphi(x) <\alpha  \}),\]
and $ \{(x,y): \ c(x,y)-\tilde \varphi(x) <\alpha  \}$ is a Borel set since $c$ and $\tilde \varphi$ are Borel measurable.
Notice, also, that 
\[\tilde \varphi (x) + \psi(y)= c(x,y), \ \forall (x,y)\in \Gamma_1, \]
and 
\[\tilde \varphi (x) + \psi(y)\leq c(x,y), \ \forall (x,y) \in p_X(\Gamma_1) \times Y, \]
As in the previous case, assigning the value $-\infty$ on a $\nu$-null set $\tilde{N}$,  we consider a Borel measurable $\tilde \psi: Y\to [-\infty, +\infty)$
such that 
\[\tilde \psi \leq \psi \ \  \mbox{and} \ \ \tilde \psi= \psi \ \nu-a.e.\]
We define $\Gamma_2:= \Gamma_1 \setminus p_Y ^{-1}(\tilde{N})$.
The set $\Gamma_2$ is Borel measurable and $c$-cyclically monotone, and $\gamma(\Gamma_2)=1$. Moreover
\[\tilde \varphi (x)+ \tilde \psi (y)=c(x,y), \ \forall (x,y)\in \Gamma_2 \]
and 
\[\tilde \varphi (x)+ \tilde \psi (y)\leq c(x,y), \ \forall (x,y)\in P_X (\Gamma_1)\times Y.\]
Since the potential needs to be defined everywhere on $X\times Y$ we let $\tilde \varphi$ be $-\infty$ outside $p_X (\Gamma_2)$.

As a result we obtain a pair of Borel measurable functions $(\tilde \varphi, \tilde \psi)$, which is $c$-splitting for the $c$-cyclically monotone set $\Gamma_2$ on which $\gamma$ is concentrated.
The possible lack of integrability of $\tilde \varphi$ and $\tilde \psi$ does not allow us to deduce the optimality of $\gamma$ in the usual way, that is integrating directly expressions $\tilde \varphi+\tilde\psi=c$ on the support of $\gamma$ and $\tilde\varphi+\tilde\psi\le c$ everywhere.  We use truncations instead. 

We fix an arbitrary transport plan $\gamma'\in\Pi(\mu,\nu)$; we may assume that $C[\gamma']<+\infty$. 
We consider truncated potentials $\tilde\varphi_n = (n\wedge (\tilde\varphi\vee -n))$ and $\tilde\psi_n=(n\wedge (\tilde\psi\vee -n))$ and denote $\xi=\tilde\varphi+\tilde\psi$, $\xi_n=\tilde\varphi_n+\tilde\psi_n$. 
Now $\xi^n\uparrow \xi^+$ on the set $\{\xi\ge 0\}$, $\xi^n\downarrow \xi^-$ on the set $\{\xi\le 0\}$ and by monotone convergence, as $n\to\infty$,
\begin{align*}
&\int_{\xi\ge 0}\xi_nd\gamma\uparrow\int_{\xi\ge 0}\xi d\gamma<+\infty,~~~\int_{\xi\ge 0}\xi_nd\gamma'\uparrow\int_{\xi\ge 0}\xi d\gamma'<+\infty,\\
&\int_{\xi\le0}\xi^n d\gamma\downarrow\int_{\xi\le0}\xi d\gamma ~~~\text{and}~~~\int_{\xi\le0}\xi_n d\gamma'\downarrow\int_{\xi\le0}\xi d\gamma'.
\end{align*}
Therefore the integrals of $\xi$ with respect to $\gamma$ and $\gamma'$ are well-defined, and
\[\lim_{n\to\infty}\int \xi_nd\gamma=\int\xi d\gamma~~~\text{and}~~~\lim_{n\to\infty}\int\xi_n d\gamma'=\int \xi d\gamma'.\] 
Note that since $\gamma$ and $\gamma'$ have the same marginals, we have 
\begin{align*}
&\int\xi d\gamma=\lim_{n\to\infty}\int\xi_n d\gamma=\lim_{n\to\infty}\left(\int\tilde\varphi_n d\mu+\int \tilde\psi_n d\nu\right)\\
&=\lim_{n\to\infty}\int \xi_n d\gamma'=\int\xi d\gamma'.
\end{align*}
The optimality of $\gamma$ now readily follows since
\begin{align*}
C[\gamma]&=\int cd\gamma=\int\xi d\gamma=\int\xi d\gamma' 
\ge\int c d\gamma'=C[\gamma']. \qedhere
\end{align*}
\end{proof}

Note that the starting point of the above proof of Theorem \ref{cFinite} is the formula \eqref{potential} which in Theorem \ref{RocRuc} gives the existence of a real-valued potential. However, as we have seen in Section \ref{sec:cyclic_mon}, when the cost attains infinite values this approach fails. Hence, some assumptions must be added to give sense to the formula, and a useful notion to consider seems to be $c$-connectivity (recall Definition \ref{connecting}). 

Another approach is to arrive at the existence of a potential without the use of a formula of the type \eqref{potential}.
In fact it was shown in \cite{artstein2022rockafellar} that if the cost attains infinite values then $c$-cyclic monotonicity is no longer equivalent to the existence of a potential (see example at the end of Section 2).  The authors established that a necessary and sufficient condition for the existence of a potential is that of $c$-path-boundedness (see Definition \ref{def:c-path-bdd}). This approach also appears in the multi marginal setting (see Theorem  \ref{thm:griessler}). 

The issue arising from the latter approach is the possible lack of measurability. This is why we are reverting to the condition of $c$-connectivity.

\begin{theo}\label{Optimal-if-connecting}
	Let $X, \,Y$ be Polish spaces equipped with Borel probability measures $\mu, \nu$, and let $c : X \times Y \to [0, \infty]$ a Borel measurable cost function. Let $\gamma$ be a $c$-cyclically monotone transport plan with finite total cost and assume that $\gamma$ is supported on a  $c$-connecting set $\Gamma$.
Then $\gamma$ is optimal.
\end{theo}

\begin{proof} 
Let us recall the Rockafellar-Rochet-R\"uschendorf formula: fix $(x_0,y_0)\in \Gamma$ and define the function
$\varphi_n:X\times \Gamma^n:(-\infty,+\infty]$ by setting 
\begin{equation}\label{formula}
\varphi_n(x,x_1,y_1,\ldots,x_n,y_n)=c(x,y_n)-c(x_n,y_n)+\sum_{i=0}^{n-1}\big(c(x_{i+1},y_i)-c(x_i,y_i)\big).
\end{equation}
Define for all $x\in X$
\begin{equation}\label{rock}
\varphi(x):=\inf\{\varphi_n(x,x_1,y_1,\ldots, x_n,y_n)~|~n\ge 1, ~(x_i,y_i)_{i=1}^n\in\Gamma^n\}.
\end{equation}

The key is to prove that also in this case the formula (\ref{rock}) makes sense. From that point on one can proceed like in the proof of Theorem \ref{cFinite}. Hence it suffices to show the following:

\begin{claim}\quad  \vspace{-0.8cm}
\begin{align}
  &\varphi(x)\in\R\text{ for all }x\in p_X(\Gamma)\ \ \text{ and }\label{claim1}\\
&\varphi(x)\le\varphi(x')+ c(x,y)-c(x',y)~~~\text{for all }x\in X\text{ and }(x',y)\in \Gamma.\label{ineq}
\end{align} \end{claim}

\textit{Proof. }
We fix $x\in p_X(\Gamma)$. So there exists $y\in Y$ such that $(x,y)\in\Gamma$. 
Since $\Gamma$ is $c$-connecting and $(x,y),(x_0,y_0)\in\Gamma$, we can find $(x_1,y_1),\ldots, (x_{n-1},y_{n-1})\in\Gamma$ such that, setting $(x_n,y_n)=(x,y)$, we have 
\[\varphi_n(x;x_1,y_1,\ldots, x_n,y_n)<+\infty.\]
This means that $\varphi(x)<+\infty$, and it remains to show that $\varphi(x)>-\infty$. 

We denote $a_1=x$. Since $x\in p_X(\Gamma)$, there exists $b_1\in Y$ such that $(a_1,b_1)\in\Gamma$. The $c$-connectivity of $\Gamma$ gives us the existence of a chain $(a_1,b_1),\ldots,(a_m,b_m)\in\Gamma$ such that $c(a_2,b_1),$ $c(a_3,b_2),\ldots, c(a_m,b_{m-1})<+\infty$. 
For any ordered collection $(x_1,y_1),\ldots,(x_n,y_n)\in\Gamma$ such that 
\[\varphi_n(x,x_1,y_1,\ldots, x_n,y_n)<+\infty\]
we extend the chain $(x_1,y_1),\ldots, (x_n,y_n)$ by setting for all $i\in\{1,\ldots, m\}$ $x_{n+i}=a_i$ and $y_{n+i}=b_i$. 
Due to the $c$-cyclic monotonicity of $\gamma$ and the finiteness of all the terms involved, we have (denoting $x_{m+n+1}=x_0$) that
\[0\le \sum_{i=0}^{n+m}\left(c(x_{i+1},y_i)-c(x_i,y_i)\right)= c(x_0,y_{n+m})-c(x_{n+m},y_{n+m})+\sum_{i=0}^{n+m-1}\big(c(x_{i+1},y_i)-c(x_i,y_i)\big).\]
That is 
\begin{align*}
\alpha:&=c(a_m,b_m)-c(x_0,b_m)+\sum_{i=1}^{m-1}\big(c(a_i,b_i)-c(a_{i+1},b_i)\big)\\ &\le c(x,y_n)-c(x_n,y_n)+\sum_{i=0}^{n-1}\big(c(x_{i+1},y_i)-c(x_i,y_i)\big).
\end{align*}
 This, in turn, means that $\alpha\le \varphi_n(x,x_1,y_1,\ldots,x_n,y_n)$, and passing to the infimum we get that  $-\infty<\alpha\le\varphi(x)$. 

To prove the inequality (\ref{ineq}), observe that its right hand side can we written as 
\begin{align*}
&\inf_{n\ge 1, (x_i,y_i)_{i=1}^n\in\Gamma^n}\left\{c(x',y_n)-c(x_n,y_n)+\sum_{i=0}^{n-1}\big(c(x_{i+1},y_i)-c(x_i,y_i)\big)\right\} +c(x,y)-c(x',y)\\
&=\inf_{n\ge 1,(x_i,y_i)_{i=1}^n\in\Gamma^n}\left\{c(x,y)-c(x',y)+\sum_{i=0}^{n-1}\big(c(x_{i+1},y_i)-c(x_i,y_i)\big)+\big(c(x',y_n)-c(x_n,y_n)\big)\right\}\\
&=\inf_{m\ge 1,(x_i,y_i)_{i=1}^m\in\Gamma^m,(x_m,y_m)=(x',y)}\left\{c(x,y_m)-c(x_m,y_m)+\sum_{i=0}^{m-1}[c(x_{i+1},y_i)-c(x_i,y_i)]\right\}\\
&=\inf_{m\ge 1,(x_i,y_i)_{i=1}^m\in\Gamma^m,(x_m,y_m)=(x',y)}\varphi_n(x,x_1,y_1,\ldots, x_m,y_m).
\end{align*}
Finally, we conclude that on both sides of \eqref{ineq} we take the infimum of the same function but on the right-hand side the set over which we minimize is smaller, which concludes the proof. 
\end{proof}

\begin{rema}
Example \ref{esempioovunque} shows that the assumption of connectivity in the above theorem is not necessary for optimality. 
\end{rema}

The previous theorem can be used as a building block to study under which conditions all $c$-cyclically monotone transport plans are optimal. This boils down to the study of the structure of the set where the cost is infinite, as it has been done in \cite{Schachermayer-optimal-and-better}, and we recall the theorem. 

\begin{manualtheorem}{18}[{\cite[Theorem 1.b]{Schachermayer-optimal-and-better}}]
Assume that the set where $c\equiv +\infty$ is the union of a closed set $F$ and a $\mu\otimes\nu$-null set $N$. Let $\gamma\in\Pi(\mu,\nu)$ be a finite and $c$-cyclically monotone plan. Then $\gamma$ is optimal.
\end{manualtheorem}

The proof is based on finding a set $\Gamma\subset X\times Y$ of full $\gamma$-measure and a partition $\{\Gamma\cap(C_i\times D_i)\}_{i\in I}$ of $\Gamma$ which is countable and such that each element of the partition is $c$-connecting. This can be done due to the fact that the set $X\times Y\setminus F$ is open. The key property of the partition is that not only the starting plan $\gamma$ but also \emph{any other finite transport plan $\tilde\gamma$} satisfies $\tilde\gamma\left(\bigcup_{i\in I}C_i\times D_i\right)=1$. 

\subsection{Existence of plans supported on a $c$-subgradient}
In \cite{artstein2023optimal} the authors took a slightly different path to finding a Brenier-type map. Omitting the equivalence between optimality and a plan being supported on a $c$-subgradient of a function, they showed that such an optimal plan exists. More precisely, we have the following theorem.  

\begin{theo}[{\cite[Theorem 1.1]{artstein2023optimal}}]\label{thm:transport-polar-comp}
	Let $X=Y$ be a Polish space, let $c: X\times Y \to \R\cup\{+\infty\}$ be a continuous and symmetric cost function, essentially bounded from below with respect to probability measures $\mu\in {\mathcal P}(X)$ and $\nu\in {\mathcal P}(Y)$. Assume $\mu$ and $\nu$  are strongly $c$-compatible, namely satisfy that for any measurable $A\subset X$ we have 
	\begin{align}\label{cond:compatibility}
	\mu(A) + \nu(\{y\in Y: \forall x\in A, \,\, c(x,y) = \infty \}) \le 1,
	\end{align}
	and for any measurable $A\subset X$ with $\mu(A)\neq 0,1$ we have
	\begin{align}\label{cond:strong-comp}
	\mu(A) + \nu(\{y\in Y: \forall x\in A, \,\, c(x,y) = \infty \}) < 1.
	\end{align}
	If there exists \emph{some} finite plan transporting $\mu$ to $\nu$, then there exists a  {$c$-class function} $\varphi$ and an optimal transport plan $\pi \in \Pi(\mu, \nu)$ concentrated on $\partial^c \varphi$.
\end{theo}

\begin{rema}
Note that the condition \eqref{cond:compatibility} is a necessary condition for the existence of a plan such that each atom is transported at finite cost. In the special discrete setting this condition corresponds to the Hall's marriage theorem that guarantees the existence of a matching. The second condition \eqref{cond:strong-comp} turns out to be a sufficient condition. It guarantees that the optimal plan, which is known to be $c$-cyclically monotone is \emph{connecting} and therefore $c$-path-bounded. 
\end{rema}

\begin{rema}
The assumption of lower semi-continuity of the cost is standard and as mentioned previously it guarantees the existence of an optimal plan. The authors assume upper semi-continuity to guarantee measurability of the set $\{y\in Y: \forall x\in A, \,\, c(x,y) = \infty \}$, which in this case is simply closed and hence measurable. However, this assumption is not needed. Indeed, assuming only the lower semi-continuity of the cost we see that the set \[\{y\in Y: \forall x\in A, \,\, c(x,y) = \infty \} = \cap_{n\in\N} \{y\in Y: \forall x\in A, \,\, c(x,y) > n\}\]  is a countable intersection of open sets and as such is measurable. Moreover, the measurability of the potential follows from the fact that the support of the plan is $c$-connecting.   
\end{rema}

Moreover, in the case when the target measure $\nu$ has finite support, it was shown (\cite[Theorem 1.3]{artstein2023optimal}) that it is sufficient to assume measurability of the cost (and a mild condition regarding a joint property of the measures and the cost called $c$-regularity, see \cite[Definition 5.1]{artstein2023optimal}).

\section{More general transport problems}
\subsection{The multi-marginal, integral case}\label{mot}
Let $N\ge 2$,  $X_1,\ldots,X_N$ be Polish spaces, denote $X=\prod_{i=1}^NX_i$ and consider a cost function $c:X\to \R\cup\{+\infty\}$. The natural generalization of the classical optimal transport to the multidimensional case is to consider
\begin{equation*}
 \min_{\gamma \in \Pi(\mu_1,\ldots,\mu_N)} C(\gamma):= \min_{\gamma \in \Pi(\mu_1,\ldots,\mu_N)} \int_X c 
 d\gamma, 
\end{equation*}
 in the set 
\[\Pi(\mu_1,\ldots,\mu_N):=\{\gamma\in\mathcal{P}(X)~|~p_{X_i}(\gamma)=\mu_i\text{ for all }i=1,\ldots,N\},\]
that is, in the set of \emph{couplings} or \emph{transport plans} between the $N$ marginals $\mu_1,\ldots,\mu_N$. 

Like in the two-marginals case, to establish the existence of minimizers it is sufficient to assume, for example, that the cost function is lower semi-continuous and essentially bounded from below. 

The notion of multidimensional $c$-cyclic monotonicity was introduced in Definition \ref{multicyclic}. It brings the concept of $c$-cyclic monotonicity in classical optimal transport to the multimarginal case. The multimarginal definition is more complex but carries the same idea: if we permute the `destinations' of points in the support of an optimal plan, the cost should not improve. 
In finite spaces for real-valued cost functions,  $c$-cyclic monotonicity and optimality are equivalent also in the multimarginal case, see Theorem \ref{CmIffOptimal} in Appendix \ref{appendix:finitely}. 

The rest of this subsection is dedicated to an overview of what we know about the sufficiency of $c$-cyclic monotonicity for optimality in general Polish spaces for arbitrary Borel probability measures as marginals. 

Recall that the concept of a $c$-splitting set (Definition \ref{cm}) is the multimarginal counterpart of the $c$-subgradient and a useful tool for proving the optimality of a plan. 
In \cite{kim2014general} Kim and Pass proved that the multimarginal  $c$-cyclic monotonicity is a necessary condition for the $c$-splitting property of a set. 
In \cite{griessler}, Griessler showed that in Polish spaces, under some boundedness conditions on a real-valued cost function, $c$-cyclical monotonicity of the support of a plan implies its optimality. More precisely, we have the following theorem. 

\begin{theo}[{\cite[Theorem 1.2]{griessler}}]\label{thm:griessler}
	Let $X_1,\ldots,X_N$ be Polish spaces and $\mu_1,\ldots,\mu_N$ be Borel probability measures on the spaces $X_1,\ldots,X_N$. Let $c:X_1\times \ldots \times X_N \to [0,\infty)$ be a continuous cost function such that for some $f_i\in L^1(X_i,\mu_i)$, $i=1,\ldots,N$, one has $c(x_1,\ldots,x_N)\le\sum_{i=1}^Nf_i(x_i)$.\footnote{This condition does not need to hold everywhere. It is enough that there exist $N$ sets $N_i\subset X_i$ such that $\mu_i( N_i)=0$ and such that the condition holds outside $\cup_i  p_{X_i}^{-1} (N_i)$. } 
	Let $\gamma \in \Pi(\mu_1, \ldots, \mu_N)$ be a transport plan such that ${\rm supp} (\gamma)$ is $c$-cyclically monotone. Then $\gamma$ is optimal. 
\end{theo}

Griessler starts by proving that a $c$-cyclically monotone transport plan is finitely optimal, i.e. that  all its finitely-supported submeasures are optimal with respect to their marginals. He continues by commenting that by a linear programming argument one can show that finite and $c$-cyclically monotone sets are $c$-splitting. The transition from finite $c$-splitting sets to arbitrary $c$-splitting sets is done by first fixing a $c$-cyclically monotone set $\Gamma$ and an element $x^0\in\Gamma$, and then proving that the family
\[\mathcal{F}:=\bigcup_{G\subset \Gamma \text{ finite }}\mathcal{G}_G\,,\]
where
\begin{align*}
\mathcal{G}_G=\{(\varphi_1,\ldots,\varphi_N): \, &\varphi \text{ is a }(G\cup\{x^0\},c)-\text{splitting tuple such that}\\
&\varphi_i(x_i)\le c(x_1^0,\ldots,x_{i-1}^0,x_i,x_{i+1}^0,\ldots, x_N^0)\text{ for all }i=1,\ldots, N\}\,,
\end{align*}
has the finite-intersection property. Hence by the compactness of the space $\overline \R^{X_1}\times\overline\R^{X_2}\times\cdots\times\overline\R^{X_N}$ the set 
\[\bigcap_{G\subset\Gamma\,,~G\text{ finite}}\mathcal{G}_{G}\]
is nonempty, proving the existence of a $(\Gamma,c)$-splitting tuple.

To prove the existence of a Borel measurable tuple Griessler uses the continuity of the cost function: if $(\varphi_1,\ldots,\varphi_N)$ is a $(\Gamma,c)$-splitting tuple, then by first setting
\[\tilde\varphi_1(x_1)=\inf\left\{c(x_1,y_2,\ldots,y_N)-\sum_{j=2}^N\varphi_j(y_j): \ (y_2,\ldots,y_N)\in X_2\times\cdots\times X_N\right\}\]
and then, assuming that $\tilde\varphi_1,\ldots,\tilde\varphi_{i-1}$ are already defined, 
\begin{multline*}
\tilde\varphi_i(x_i)=\inf\{
c(y_1,\ldots,y_{i-1},x_i,y_{i+1},\ldots,y_N)-\sum_{j=1}^{i-1}\tilde\varphi_j(y_j)-\sum_{j=i+1}^N\varphi(y_j): \\ 
y_j\in X_j,~j\in\{1,\ldots,i-1,i+1,\ldots, N\}\}
\end{multline*}
one obtains a $(\Gamma,c)$-splitting tuple that is upper semicontinous and, in particular, Borel measurable. 

To pass from measurable potentials to potentials whose integrals are well-defined (and thus to optimality) Griessler shows that if $(\varphi_1,\ldots,\varphi_N)$ is a $(\Gamma,c)$-splitting tuple and if $x^0\in\Gamma$, then there exists a $(\Gamma, c)$-splitting tuple $(\tilde\varphi_1,\ldots,\tilde\varphi_N)$ such that 
\begin{equation}\label{cbound}
\tilde\varphi_i(x_i)\le c(x_1^0,\ldots,x_{i-1}^0,x_i,x_{i+1}^0,\ldots,x_N^0)~~\text{for all }i\in\{1,\ldots,N\}\text{ and }x_i\in X_i\,.
\end{equation}
This is achieved by first noting that, since $x^0\in\Gamma$ and since $(\varphi_1,\ldots,\varphi_N)$ is a $(\Gamma,c)$-splitting tuple, we have 
\[\sum_{i=1}^N\varphi_i(x_i^0)=c(x_1^0,\ldots,x_N^0)\,,\]
and, in particular, (since $c$ assumes only real values) $\varphi_i(x_i^0)\in\R$ for all $i\in\{1,\ldots,N\}$. 
Then one defines 
\begin{align*}
\tilde\varphi_1(x_1)&=\varphi_1(x_1)+\sum_{j=2}^N\varphi_j(x_j^0)~~~\text{and}\\
\tilde\varphi_i(x_i)&=\varphi_i(x_i)-\varphi_i(x_i^0)~\text{ for }i=2,\ldots,N.
\end{align*}
Starting from a Borel measurable $(\Gamma,c)$-splitting tuple $(\varphi_1,\ldots,\varphi_N)$ one gets the existence of a Borel measurable $(\Gamma,c)$-splitting tuple $(\tilde\varphi_1,\ldots,\tilde\varphi_N)$ satisfying, in addition,  \eqref{cbound}. The fact that the integrals of the functions $\tilde\varphi_i$  are well-defined follows by the assumption that $c$ is bounded from above by a sum $\sum_{i=1}^Nf_i$ of $\mu_i$-integrable functions: for a fixed $x^0\in\Gamma$ that satisfies $\sum_{i=1}^Nf_i(x_i^0)<+\infty$ we have 
\begin{align*}
\int_{X_i}\tilde\varphi_i(x_i)d\mu_i(x_i)&\le \int_{X_i}c(x_1^0,\ldots,x_{i-1}^0,x_i,x_{i+1}^0,\ldots,x_N^0)d\mu_i(x_i)\\ &\le \sum_{j\ne i}f_j(x_j^0)+\int_{X_i}f_i(x_i)d\mu_i(x_i)<+\infty.
\end{align*}
If $\Gamma$ is a $c$-cyclically monotone set on which a transport plan $\gamma\in\Pi(\mu_1,\ldots,\mu_N)$ is concentrated, the the optimality of the plan $\gamma$ now follows easily: if $\gamma'$ is another transport plan, one has (denoting $X=\prod_{i=1}^NX_i$) that
\[ \int_X cd\gamma=\int_X\sum_{i=1}^N\tilde\varphi_id\gamma=\int_X\sum_{i=1}^N\varphi_id\gamma'\le\int_Xc d\gamma'\,.\]

It is an open problem if the $c$-cyclic monotonicity of the support of a plan implies its optimality if some of the assumptions of Griessler are lifted. For example if we consider real-valued, Borel (not necessarily continuous) cost function bounded from above by a sum of integrable functions. For this type of a cost one could prove the existence of a $(c,\Gamma)$-splitting tuple as above, but without the continuity assumption on the cost the existence of a Borel measurable splitting tuple is not clear. Note that in the $2$-marginal case the measurability was based on the Rockafellar-type formula \eqref{potential}, but in the multimarginal case this kind of formula is not available. 

If the cost function assumes the value $+\infty$ there are results in the opposite direction.  In \cite{petrache2020cyclically}, M. Petrache presented an example, for the three-marginal case, of a $c$-cyclically monotone transport plan that is not optimal. In the example the spaces $X_i$ were all equal to $\N$, the natural numbers, and the cost function assumes infinite value. Clearly, this discrete example can naturally be embedded in $\R^3$ where we view the discrete marginal measures as elements of $\mathcal{P}(\R)$. Note that in this case the cost is not continuous. While if we regard $\N$ as a Polish space equipped with the discrete metric, the cost is continuous but not bounded from above by a sum of integrable functions of the marginal measure spaces. 

\begin{exam}[{\cite[Proposition 2.1]{petrache2020cyclically}}] We fix $N=3$ and consider the three-marginal optimal transportation problem with the spaces $X_1=X_2=X_3=\N$ and all marginals equal to the measure 
\[\mu=\sum_{k=1}^\infty 2^{-k}\delta_k\in \mathcal{P}(\N).\]
To define the cost function we fix an auxiliary function $f:\N\to (0,1]$ such that 
\begin{equation}\label{condition_f}
\sum_{k=1}^\infty 4^{-k}(f(2k-1)-f(2k))>\frac16.
\end{equation}
This condition is satisfied, for example, by all decreasing functions $f:\N\to (0,1]$ such that $f(1)>~f(2)+\frac23$. 
We define the cost function $c:(\N)^3\to [0,\infty]$ as follows:
\[\begin{cases}
c(a,b,c) \text{ is symmetric with respect to permuting the coordinates }(a,b,c),\\
c(1,1,1)=1,\\
c(a,a,a+1)=f(a)\text{ for all }a\in\N\text{ and }\\
c=+\infty\text{ for all triples }(a,b,c)\text{ not described by the previous two conditions.}
\end{cases}\]
We consider the following transport plans $\gamma,\bar\gamma\in\Pi(\mu,\mu,\mu)$:
\begin{align*}
\gamma&=\sum_{k=1}^\infty 4^{-k}(\delta_{(2k-1,2k-1,2k)}+\delta_{(2k-1,2k,2k-1)}+\delta_{(2k,2k-1,2k-1)})~~~\text{and}\\
\bar\gamma&=\frac12\delta_{(1,1,1)}+\frac12\sum_{k=1}^\infty 4^{-k}(\delta_{(2k,2k,2k+1)}+\delta_{(2k,2k+1,2k)}+\delta_{(2k+1,2k,2k)}).
\end{align*}
We will show that the plan $\gamma$ is \emph{$c$-cyclically monotone but not optimal}. 

The plan $\gamma$ cannot be optimal since the plan $\bar\gamma$ has a lower cost:
\begin{align*}
C[\gamma]-C[\bar\gamma]&=3\sum_{k=1}^\infty 4^{-k} f(2k-1)-\frac12-3\sum_{k=1}^\infty 4^{-k}f(2k)\\
&=3\sum_{k=1}^\infty 4^{-k}(f(2k-1)-f(2k))-\frac12>0,
\end{align*}
where we have used the property (\ref{condition_f}) of the function $f.$

It remains to show that $\gamma$ is $c$-cyclically monotone. We show that it is finitely-optimal; by Propositions \ref{icmfinitelyoptimal} and \ref{finitelyOptimalCm} of Appendix \ref{sec:App finite optimality} these two notions are equivalent. 
We fix a finitely-supported subplan $\alpha$ of $\gamma$, and another transport plan $\alpha'$ that has the same marginals as $\alpha$.  It is enough to show that $C[\alpha]\le C[\alpha']$. If $C[\alpha']=+\infty$, the claim holds trivially. We may thus assume that $C[\alpha']<+\infty$. We show that in fact in this case $\alpha=\alpha'$; the idea is that the set where for given marginals the cost can be finite is very small. 

In the following we denote by $(\delta_{(a,b,c)})^S$ the symmetrisation
\[(\delta_{(a,b,c)})^S=\frac16(\delta_{(a,b,c)}+\delta_{(a,c,b)}+\delta_{(b,a,c)}+\delta_{(b,c,a)}+\delta_{(c,a,b)}+\delta_{(c,b,a)}).\]
Due to the symmetries of the problem and the structure of the plan $\gamma$ we may assume that for some $M>0$ and coefficients $a_k\ge 0$
\[\alpha=\sum_{k=0}^M a_k (\delta_{(2k-1,2k-1,2k)})^S.\]
The marginals of this plan are all equal to 
\[\mu_\alpha=\sum_{k=1}^Ma_k\left(\frac23\delta_{2k-1}+\frac13\delta_{2k}\right).\]
We may also assume that the measure $\alpha'$ is symmetric. 
Since $C[\alpha']<+\infty$, there exists nonnegative constants $a',a_1',\ldots,a_M'$, $b_1',\ldots,b_M'$ such that
\[a'+\sum_{k=1}^M(a_k'+b_k')=1~~~\text{and}~~~\alpha'=a'\delta_{(1,1,1)}+\sum_{k=1}^M(a_k'(\delta_{(2k-1,2k-1,2k)})^S+b_k'(\delta_{(2k,2k,2k+1)})^S).\]
Since the marginals of $\alpha'$ are all equal to $\mu_\alpha$, we have
\[\mu_\alpha(\{1\})=\frac23a_1=\alpha'(\{1\}\times\N\times\N)=a'+\frac23a_1'\]
and thus
\begin{equation}\label{e1}
2a_1=3a'+2a_1'.
\end{equation}
Similarly we have, for all $k\ge 2$:
\[\mu_\alpha(\{2k-1\})=\frac23a_k=\alpha'(\{2k-1\}\times\N\times\N)=\frac23a_k'+\frac13b_{k-1}'\]
which implies that
 \begin{equation}\label{e2}
2a_k=2a_k'+b_{k-1}'~~~\text{for all } k\ge 2.
\end{equation}
And finally, for $k\ge 1$, we have
\[\mu_\alpha(\{2k\})=\frac13a_k=\alpha'(\{2k\}\times\N\times\N)=\frac13a_k'+\frac23b_k'\]
giving
\begin{equation}\label{e3}
a_k=a_k'+2b_k'~~~\text{for all }k\ge 1.
\end{equation}
Equations (\ref{e2}) and (\ref{e3}) imply that
\[a_k'+2b_k'=a_k'+\frac{b_{k-1}'}{2}~~~\text{for all }k\ge 2,\]
and thus 
\[b_{k+1}'=\frac{b_{k}'}{4}~~~\text{for all }k\ge 1.\]
Since $b_k=0$ for all $k>M$ we have, therefore, $b_k=0$ for all $k=1$. 
Hence, by (\ref{e3}), we have $a_k=a_k'$ for all $k\ge 1$, and this inserted in (\ref{e1}) gives $a'=0$, completing the proof of $\alpha=\alpha'$. 
\end{exam}

The counterexample above shows that neither of the two-marginal Theorems \ref{pratelliatomic} and \ref{thm:Partelli cont cost} of Pratelli can be extended to the multi-marginal case.

\subsection{The $L^\infty$-problem. }
Let $X_1$ and $X_2$ be Polish spaces and $c:X_1\times X_2\to [0,\infty]$  be a lower semi-continuous cost function.
The \emph{$L^\infty$ optimal mass transportation problem} can be stated as follows:
\[ \min_{ \Pi(\mu_1,\mu_2)} C_\infty[\gamma]:= \min_{ \Pi(\mu_1,\mu_2)}  \gamma-\esup c(x_1,x_2)
\tag{$P_\infty$}
\]
This problem is `a limit case' of the classical optimal transport since, if $c\ge 0$, is the variational limit, as $p\to\infty$, of the optimal transport problems with the cost functional $C_p[\gamma]:= \left(\int_Xc^p \,d\gamma\right)^\frac 1p$. In fact, if we take a sequence $(\gamma_p)_{p=1}^\infty$, with $\gamma_p$  minimizer of the corresponding functional $C_p$ (and thus $c$-cyclically monotone), 
we get, as the weak${}^\ast$-cluster points $\gamma=\lim_{p\to\infty}\gamma_p$, some solutions of the problem ($P_\infty$) corresponding to the cost function $c$. If $c$ is continuous and if for each $p$ the value $C_p[\gamma]$ is finite, these solutions are \emph{infinitely cyclically monotone}, in the sense of the following definition:
\begin{defi}
We say that a set $\Gamma\subset X_1\times X_2$ is \emph{infinitely $c$-cyclically monotone}, if for every $k$-tuple of points $(x_1^i,x_2^i)_{i=1}^k$  of $\Gamma$ and for every permutation $\sigma$ of the set $\{1,\ldots, k\}$ we have
\[\max\{c(x_1^i,x_2^i)~|~i\in \{1,\ldots,k\}\}\le \max\{c(x_1^i,x_2^{\sigma(i)})~|~i\in \{1,\ldots,k\}\}\,.\]
We also say that a coupling $\gamma\in \Pi(\mu_1,\mu_2)$ is infinitely cyclically monotone if it is concentrated on an \ICM set. 
\end{defi}
The \ICM transport plans are a class of `well-behaved' solutions to the problem ($P_\infty$). These plans are well-behaved in the sense that the restrictions of \ICM plans are still \ICM and thus optimal (if \ICM plans are optimal). In general, the problem ($P_\infty$) is much more cumbersome than the standard optimal transportation problem. The cost functional is neither linear nor convex, there is no classical duality theory, and there is a high degree of non-uniqueness of the solutions. Loosely speaking, in the problem ($P_\infty$) we are only interested in optimizing the worst case, so locally its solutions can be far from optimal. In this last sense, the \ICM minimizers have better properties, and, at least for continuous cost $c$ that admits finite $C_p$ solutions, the \ICM couplings exist as stated above. 

The question we consider is: are optimal plans infinitely cyclically monotone, and are \ICM plans optimal? 
The classical two approaches to the necessity of the condition (see Section \ref{sec:optimal-ccm}) are based on the linearity or on the duality for convex problems, but as was mentioned this formulation of the problem lacks both. 
In fact, not all plans that are optimal for  ($P_\infty$) are infinitely cyclically monotone. The next example illustrates the typical behavior. By $\mathcal{L}^1|_{[0,1]}$ we denote one-dimensional Lebesgue measure on the interval $[0,1]$.  
\begin{exam}
Let $X_1=X_2=\R$, $c(x,y)=|x-y|$, $\mu=\tfrac12(\mathcal{L}^1|_{[0,1]}+\delta_{1})$, and $\nu=\tfrac12(\mathcal{L}^1|_{[0,1]}+\delta_{10})$, then any plan that sends $1$ to $10$ is optimal, even if the interval $[0,1]$ was mapped to itself in a very ``non-optimal" way (for example by a map $x\mapsto 1-x$). The locally ``best-behaved" optimal transport plan would be $\gamma=\tfrac12((id,id)_\sharp\mathcal{L}^1|_{[0,1]}+\delta_{(1,10)})$, and this plan is, actually, infinitely cyclically monotone.
\end{exam}
Concerning the sufficiency,  Champion, De Pascale, and Juutinen in \cite{champion2008wasserstein} laid the foundations for the two-marginal $L^\infty$ transport on $\R^d$ with the Euclidean distance as a cost function.
 Later Jylh\"a in \cite{jylha2015optimal} proved the corresponding result in Polish spaces for continuous cost functions $c:X_1\times X_2\to[0,\infty]$.   If the continuity assumption is dropped, the optimality of an \ICM plan fails even in the case of a lower semi-continuous cost with finite values. The following counterexample is a slightly modified version of the one given by Pratelli and Ambrosio in \cite{ambrosio-pratelli} of a non-optimal $c$-cyclically monotone transportation plan. In the original counterexample it was necessary that $c$ assumes the value $+\infty$ (since for a lower semi-continuous cost that assumes only finite values, in the integral case, $c$-cyclic monotonicity is equivalent to optimality). 
\begin{exam}
Consider the two-marginal $L^\infty$-optimal transportation problem with marginals $\mu=\nu=\mathcal{L}|_{[0,1]}$ and the cost function
\[c(x,y)=
\begin{cases} 
1&\text{ if }x=y\\
2&\text{ otherwise.}
\end{cases}\]
Let $\alpha$ be and irrational number. Set $T_1=Id_{[0,1]}$ and $T_2:[0,1]\to[0,1]$, $T_2(x)=x+\alpha\pmod 1$. The map $T_1$ is an optimal transportation map for the problem ($P_\infty$) with $C_\infty[T_1]=1$. Since the map $T_2$ has the cost $2$, it cannot be optimal. However, it is infinitely cyclically monotone, as we show below. 
Assume, contrary to the claim, that $T_2$ is not infinitely cyclically monotone. Then there exists a minimal $K\in\N$ and a $K$-tuple of couples $\{x_i,y_i\}_{i=1}^K $, all belonging to the support of the plan given by $T_2$, such that 
\begin{equation*}
\max_{1\le i\le K}c(x_i,y_i)>\max_{1\le i\le K}c(x_{i+1},y_i),
\end{equation*}
with the convention $x_{K+1}=x_1$. 
By the definition of the map $T_2$ we have $y_i=x_i+\alpha\pmod 1$ for all $i$. 
By the definition of the cost $c$, the only form in which this inequality can hold is: $2>1$. The right-hand side  now tells us that $y_i=x_i+\alpha\pmod 1$ for all $i$, that is, $x_{i+1}=x_i+\alpha\pmod 1$ for all $i$. Summing up, this gives (keeping in mind that $x_{K+1}=x_1$) that $x_1=x_1+K\alpha\pmod 1$, contradicting the irrationality of $\alpha$.
\end{exam}

\subsubsection{The multi-marginal case }
The $L^\infty$ problem can be stated analogously also in the multi-marginal case:
For  $N\ge 2$, let $(X_1,\mu_1),\ldots, (X_N,\mu_N)$ be Polish probability spaces equipped with Borel probability measures $\mu_i\in~ \mathcal{P}(X_i)$. Denoting  $X=\prod_{i=1}^NX_i$, we consider the cost function $c:X\to [0,\infty]$ and the following minimization problem
\[
 \min_{\gamma \in \Pi(\mu_1,\ldots,\mu_N)}C_\infty [\gamma]:= \min_{\gamma \in \Pi(\mu_1,\ldots,\mu_N)} \gamma-\esup_{(x_1,\ldots,x_N)\in X} \, c,\tag{P$_\infty^N$}
\]
 where 
\[\Pi(\mu_1,\ldots,\mu_N):=\{\gamma\in\mathcal{P}(X):\,p_{X_i}(\gamma)=\mu_i\text{ for all }i=1,\ldots,N\}.\]
In this context the $c$-cyclic monotonicity takes the following form (compare with Definition \ref{cm}): 
\begin{defi}\label{icm}
We say that a set $\Gamma\subset\prod_{i=1}^NX_i$ is \emph{infinitely $c$-cyclic monotone} if for every $k$-tuple of points $(x^{1,i},\ldots, x^{N,i})_{i=1}^k$ of $\Gamma$ and every $(N-1)$-tuple of permutations $(\sigma_2,\ldots,\sigma_N)$ of the set $\{1,\ldots, k\}$ we have
	\[\max\{c(x_1^{i} ,x_2^{i},\ldots,x_N^{i}):\,i\in \{1,\ldots,k\}\}\le \max\{c(x_1^{i},x_2^{\sigma_2(i)},\ldots,x_N^{\sigma_N(i)}): \, i\in \{1,\ldots,k\}\}\,.\]
	We also say that a coupling $\gamma\in \Pi(\mu_1,\ldots,\mu_N)$ is infinitely cyclically monotone if it is concentrated on an \ICM set. 
\end{defi}
For $N=2$ this definition coincides with the \ICM condition presented above for the two-marginal case, and exactly like in the two-marginal $L^\infty$-optimal transportation problem, also here an optimal plan is not necessarily infinitely cyclically monotone. De Pascale and Kausamo proved in \cite{de2022sufficiency} the sufficiency of the infinite cyclic monotonicity for continuous costs and marginal measures with compact supports:
\begin{theo}[{\cite[Theorem 1.3]{de2022sufficiency}}]\label{icmoptimal}
	Let  $\mu_i \in \Pb(X_i)$ with compact support for $i=1,\dots,N$, let $c:X \to \R\cup\{+\infty\}$ be continuous. If $\gamma\in\Pi(\mu_1,\ldots,\mu_N)$ is an \ICM plan for $c$, then $\gamma$ is optimal. 
\end{theo}
Again, since in this case there is no duality available, one cannot rely on $c$-splitting $N$-tuples.  Therefore, the optimality proofs have to be based on different techniques. For example, De Pascale and Kausamo first prove that all finitely-supported submeasures of an \ICM transport plan are optimal with respect to their marginals. Using a suitable finite discretisations of the original \ICM transport problem, they then apply $\Gamma$-convergence techniques to pass from optimality of the discretised transport plan to the optimality of the full transport plan. The key property of the \ICM condition of a support of a transport plan is that it is inherited by subsets.


\appendix

\section{Additional example}\label{appendix:example}

The following example is a version of {\cite[Example 3.1]{ambrosio-pratelli}}, also presented in this note in Example \ref{esempioovunque}, but the main focus is shifted to show that $c$-path-bounded sets need not be $c$-connecting. 

\begin{exam}\label{ex:connecing+c-pbdd}
	Let $X=Y=[0,1]$ with the identification of the end points $0$ and $1$, and let $\mu=\nu$ be the Lebesgue measure on $[0,1]$. Fix an irrational number $\alpha \in [0,1)$ and define the cost function $c:X\times Y\to [0,\infty]$ as
\[
c(x,y)=\begin{cases}
1&\text{ if }x=y\\
2&\text{ if }x=y+\alpha\pmod 1\\
+\infty&\text{ otherwise. }
\end{cases}
\] Consider the set
	\[ \Gamma=\{(x,x):x\in X\}.\]
Then the set $\Gamma$ is $c$-path-bounded but not $c$-connecting. 

To see that $\Gamma$ is $c$-path-bounded, fix points $(x,y),(z,w)\in\Gamma$. For any $m\in \N$ and any chain $\{(x_i,y_i)\}_{i=2}^{m-1}$ of points in $\Gamma$, letting $(x_1,y_1)=(x,y)$ and $(x_m,y_m)=(z,w)$, we have 
\[ \sum_{i=1}^{m-1} \big(c(x_i, y_i)-c(x_{i+1}, y_{i})\big) =\sum_{i=1}^{m-1} \big(1-c(x_{i+1}, y_{i})\big)\le 0,  \]   
since for all $i\in\{1,\ldots,m-1\}$ we have $c(x_{i+1},y_i)\in\{1,2,+\infty\}$. Hence, for any start and finish points, we may choose the constant $M((x,y),(z,w))$ in Definition \ref{def:c-path-bdd} to be $0$. 

The set $\Gamma$ is not connecting: if it were connecting then, for example, for points $(0,0),(\frac12,\frac12)\in\Gamma$ there must exist a chain
$(x_1,y_1), \dots, (x_n,y_n)\in\Gamma$ with $(x_1,y_1)=(0,0)$ and $(x_n,y_n)= (\frac12,\frac12)$ such that $c(x_1,y_0), \dots, c(x_n,y_{n-1})< +\infty$. This would imply the existence of $0\le m\le n$ such that $ \frac12\equiv 0+m\alpha\pmod 1$, which is impossible given the irrationality of $\alpha$. 

\end{exam}

\section{Finite spaces}\label{appendix:finitely}
In this appendix we establish the equivalence of $c$-cyclic monotonicity and optimality in finite spaces for real-valued cost functions. Equivalently one could consider finitely-supported measures on arbitrary spaces. In the following,  $X_1,\ldots,X_N$ are finite spaces and $\mu_1,\ldots,\mu_N$ probability measures on  $X_1,\ldots,X_N$, respectively.  We denote $X=\prod_{i=1}^NX_i$.
\begin{theo}\label{CmIffOptimal}
 Let $\gamma\in\Pi(\mu_1,\ldots,\mu_N)$ be an optimal transport plan with respect to a cost function $c:X\to \R$ where $X=\prod_{i=1}^NX_i$ and $X_i$ are finite spaces. Then $\gamma$ is $c$-cyclically monotone if and only if it is optimal. 
\end{theo}

\begin{proof} Assume that $\gamma$ is optimal. 
Since we work in finite spaces we denote $X_i=\{x_i^j\}_{j\in\{1,\ldots,n_i\}}$ for $i\in\{1,\ldots,N\}$, where $n_i$ is the cardinality of $X_i$. We also denote $n=n_1\cdot\ldots\cdot n_N$ and 
\[I=\{1,\ldots, n_1\}\times\{1,\ldots,n_2\}\times\cdots\times\{1,\ldots, n_N\}.\]
\newcommand{\ii}{\mathbf{i}}
\newcommand{\xx}{\mathbf{x}}
The index vectors $(i_1,\ldots,i_N)\in I$ will be abbreviated by $\ii$ and vectors $(x_1^{i_1},\ldots,x_N^{i_N})\in X$ by $\xx^\ii$. 
Using this notation, the plan $\gamma$ can be expressed as
\[\gamma=\sum_{\ii\in I}a^\ii \delta_{\xx^\ii}~~~\text{for some coefficients }a^\ii\ge 0. \]

Let us assume, contrary to the claim, that $\gamma$ is not $c$-cyclically monotone. This means that there exist  $k\in\N$, index set $\tilde I:=\{(i_1^j,i_2^j,\ldots,i_N^j)\}_{j=1}^k\subset I$, points $\{(x_1^{i_1^j},\ldots,x_N^{i_N^j})\}_{j=1}^k\subset\text{supp}\gamma$, and permutations $\sigma_2,\ldots,\sigma_N$ of the set $\{1,\ldots, k\}$ such that 
\begin{equation}\label{anticycl}
\sum_{1\le j\le k}c(\xx^{\sigma(j)})<\sum_{1\le j\le k}c(\xx^{\ii^j}),
\end{equation}
where we have abbreviated 
\[\xx^{\ii^j}:=(x_1^{i_1^j},\ldots,x_N^{i_N^j})~~~\text{and}~~~\xx^{\sigma(j)}:=(x_1^{i_1^j},x_2^{i_2^{\sigma_2(j)}},\ldots,x_N^{i_N^{\sigma_N(j)}}).\]
We denote
\[a:=\min\{a^\ii~|~\ii\in \tilde I\}>0\]
and define
\[\tilde\gamma:=\sum_{\ii\in I\setminus \tilde I}a^\ii\delta_{\xx ^\ii}+\sum_{\ii\in \tilde I}(a^\ii-a)\delta_{\xx^\ii}+a\sum_{1\le j\le k}\delta_{\xx^{\sigma(j)}}.\]
The measure $\tilde\gamma$ has the same marginals as $\gamma$ because $\sigma_2,\ldots,\sigma_N$ are permutations. 
Using (\ref{anticycl}) we get
\begin{align*}
C[\tilde\gamma]&=\sum_{\ii\in I\setminus \tilde I}a^\ii c(\xx^\ii)+\sum_{\ii\in \tilde I}(a^\ii-a)c(\xx^\ii)+a\sum_{1\le j\le k}c(\xx^{\sigma(j)})\\
&<\sum_{\ii\in I\setminus \tilde I}a^\ii c(\xx^\ii)+\sum_{\ii\in \tilde I}(a^\ii-a)c(\xx^\ii)+a\sum_{1\le j\le k}c(\xx^{\ii ^j})=C[\gamma],
\end{align*}
contradicting the optimality of $\gamma$. 

The proof of the opposite direction ($c$-cyclical monotonicity implies optimality) is essentially the same as the proof of Proposition \ref{icmfinitelyoptimal} and we do not repeat it here. 
\end{proof}

\section{Universally measurable sets and functions}\label{sec:appendix universally}
Let $X$ be a Polish space. Consider the measure space $(X, \Sigma, \mu)$ with $\mu$ a $\sigma$-finite Borel measure (so $\B (X) \subset \Sigma$). Let $\zero_\mu=\left\{N \subset X \ : \ \mu(N)=0\right\}$, 
then $\tilde \Sigma_\mu := \left\lbrace A \cup N \ : A \in \Sigma,\ N \in \zero_\mu\right\rbrace $ is still a $\sigma$-algebra.
The measure 
 \begin{align}\label{def:completion}
 \tilde \mu: \tilde \Sigma _\mu \to [0,+\infty]
\end{align}  
\[ A \cup N \mapsto \mu(A)\]
is called the \emph{completion} of $\mu$. 

A set $B \subset X$ is called \emph{universally measurable} if it is measurable for the completion of every finite Borel measure $\mu$ on $X$. 

\begin{prop}
The set of universally measurable sets is a Borel $\sigma$-algebra.
\end{prop}
\begin{proof} 
This follows from the fact that the defining properties of a $\sigma$-algebra are stable for the intersection.
\end{proof}
If $Y$ is another Polish space and $T:X \to Y$ is a Borel map, then the image of a Borel set through $T$ is called an \emph{analytic} set. The following theorem is relevant here but we do not include the proof as it  is beyond the scope of this paper. 

\begin{theo}
Every analytic set is universally measurable.
\end{theo}

\section{Finite optimality}\label{sec:App finite optimality}
In this section we prove that the $c$-cyclic monotonicity of a transport plan is equivalent to it being finitely optimal. We present the proofs directly in the multi-marginal case and for both the integral and the $L^\infty$ optimal transport. 

We start with the definition of the finite optimality.
\newcommand{\Q}{\mathbb{Q}}
\begin{defi}\label{finop}
	Let $\gamma$ be a positive and finite Borel measure on $X$. We say that $\gamma$ is \emph{finitely optimal} if all its finitely-supported submeasures are optimal with respect to their marginals. By submeasure we mean any probability measure $\alpha$ satisfying $\, {\rm supp}(\alpha)\subset {\rm supp}(\gamma)$. 
\end{defi}

\begin{prop}\label{icmfinitelyoptimal}
	If $\gamma\in\Pi(\mu_1,\ldots, \mu_N)$ is $c$-cyclically monotone or infinitely cyclically monotone, then it is finitely optimal for the corresponding optimal transportation problem (integral or $L^\infty$, respectively). 
\end{prop} 

\begin{lemm}\label{permutations}
	Let $\alpha=\sum_{i=1}^lm_i\delta_{( x^{1,i},\ldots, x^{N,i})}$ and $\overline\alpha=\sum_{i=1}^{\overline l}\overline{m}_i\delta_{(\xbar^{1,i},\ldots,\xbar^{N,i})}$ be two discrete measures with positive, integer coefficients and the same marginals. Let us denote by $\tilde l=m^1 + \dots +m^l$ the number of rows of the following table 
$$	
\left. \begin{array}{lll}  
x^{1,1} & \dots & x^{N,1} \\
\vdots &   & \vdots \\
x^{1,1} & \dots  & x^{N,1}
\end{array}
\right \rbrace
\begin{array}{l}
\\
m^1 \mbox{- times}\\
\\
\end{array} 
$$
$$ \begin{array}{llllllllll}  
\dots & \dots & \dots & & & & & & & 
\end{array}
$$ 
$$
\left.\begin{array}{lll}  
x^{1,l} & \dots & x^{N,l} \\
\vdots &   & \vdots \\
x^{1,l} & \dots  & x^{N,l}
\end{array}
\right \rbrace
\begin{array}{l}
\\
m^l \mbox{- times}\\
\\
\end{array} 
$$
where the first $m^1$ rows are equal among themselves, the following $m^2$ rows are equal among themselves and so on.
Let $\overline A$ be the analogous table associated to $\overline \alpha$. Then $\overline A$ has $\tilde l$ rows and  there exist $(N-1)$ permutations of the set $\{1,\ldots,\tilde l\}$ such that 
$\overline{A}$ is equal to
$$	
\begin{array}{lll}  
x^{1,1} & \dots & x^{N,\sigma^N(1)} \\
\vdots &   & \vdots \\ 
x^{1,1} & \dots & x^{N,\sigma^N (m_1)}\\
x^{1,2} & \dots & x^{N, \sigma^N(m_1 +1)}\\
\vdots &   & \vdots \\ 
x^{1,l} & \dots & x^{N,\sigma^N(m_1+\dots+m_{l-1}+1)}\\
\vdots &   & \vdots  \\ 
x^{1,l} & \dots & x^{N, \sigma^N (\tilde l)}.
\end{array} 
$$
\end{lemm}

\begin{proof}
	For each $k\in\{1,\ldots, N\}$, the $k$-th marginal of $\alpha$ is given by the sum of the Dirac masses centered on the points of the $k$-th column of the table $A$ with multiplicity. Analogously, the $k$-th marginal of $\overline \alpha$ is given by the sum of the Dirac masses centered at the points of the $i$-th column of the table $\overline A$ with multiplicity. Since the marginals of $\alpha$ and $\overline \alpha$ are the same, each point $x^{k,i}$ appearing in the  $k$-th marginal must appear in both matrices the same number of times, proving the existence of the bijections $\sigma^2,\ldots,\sigma^N$ as required. This also implies that $\overline A$ has $\tilde l$ rows.  
\end{proof}

\begin{proof}[Proof of Proposition \ref{icmfinitelyoptimal}] We fix a finitely-supported submeasure $\alpha=\sum_{i=1}^l a_i\delta_{X^i}$ of $\gamma$. We need to show that $\alpha$ is an optimal coupling of its marginals. To do this, we fix another coupling, $\overline\alpha= \sum_{i=1}^{\overline l} \overline a _i\delta_{\overline X^i}$, with the same marginals as $\alpha$. We have to show that 
\begin{equation}\label{eq:costclaim}
\tilde C[\alpha]\le \tilde C[\overline \alpha],
\end{equation}
where $\tilde C$ is any of the two total costs under consideration.

Let us first assume that the discrete measures $\alpha$ and $\overline \alpha$ have rational coefficients. We consider the measures $M\alpha$ and $M \overline \alpha$, where $M$ is the product of the denominators of the coefficients of $\alpha$ and $\overline \alpha$. They are discrete measures having positive, integer coefficients and the same marginals, so we can apply Lemma \ref{permutations} to find permutations $\sigma^2,\ldots,\sigma^N$ such that $M \alpha$ and $M \overline \alpha$ have representations $A$ and $\overline A$, respectively. If $\tilde C=C$ we have, using the $c$-cyclical monotonicity of $\alpha$
\[MC[\alpha]=\sum_{i=1}^{\tilde l}  c(x^{1,i},\ldots, x^{N,i})\le \sum_{i=1}^{\tilde l} c(x^{1,i},x^{2,\sigma^2(i)},\ldots, x^{N, \sigma^N(i))})=MC[\overline \alpha],\]
proving the optimality of $\alpha$. 
If $\tilde C=C_\infty$, the conclusion is also immediate:
\[C_\infty[\alpha]=\max_{1\le i\le \tilde k}c(x^{1,i},\ldots, x^{N,i})\le \max_{1\le i\le \tilde k}c(x^{1,i},x^{2,\sigma^2(i)},\ldots, x^{N,\sigma^N(i)})= C_\infty[\overline \alpha].\]

Now, assume that $\alpha$ and $\overline \alpha$ have real (not necessarily rational) coefficients, and that
\[\alpha:= \sum_{i=1}^l a_i \delta_{X^i}, \ \ \ \overline \alpha:= \sum_{i=1}^{\overline l} \overline a_i \delta_{\overline X ^i}.
\] 
We show that for all $\varepsilon >0$ there exist two discrete measures 
\[\beta:= \sum_{i=1}^l q_i \delta_{X^i} \ \mbox{and} \  \ \overline \beta:= \sum_{i=1}^{\overline l} \overline q_i \delta_{\overline X ^i},
\] 
with the same marginals, $q_i, \overline q_i \in \Q$ and 
\[|a_i-q_i| <\varepsilon, \ \ \ |\overline a_i-\overline q_i| <\varepsilon.
\]
Being concentrated on $X^1,\dots,X^l$ and $\overline X^1, \dots \overline X^{\overline l}$ is equivalent to the fact that the vector ${\bf \underline a}:=(a_1, \dots, a_l, \overline a_1, \dots, \overline a_{\overline l})$ is a solution of 
\[\mathcal A {\bf \underline a}=0,\]
where $\mathcal A$ is a matrix with coefficients $1, 0, -1$.
Indeed, if we write, for example, the equality between the first two marginals we obtain
\[ \sum_{i=1}^l a_i \delta_{x^{1,i}}= \sum_{i=1}^{\overline l} \overline a_i \delta_{\overline x ^{1,i}}.
\] 
Therefore, some of the points $\overline x^{1,i}$ must coincide with, for example, $x^{1,1}$ which implies that for two sets of indices we have 
\[\sum_{i\in I} a_i =  \sum_{j\in J} \overline a_j.\] 
Since the matrix $\mathcal A$ has integer coefficients
\[\overline{Ker_\Q \mathcal A} = Ker_\R \mathcal A,\] 
and this allows to choose $\beta$ and $\overline \beta$.
Since $C[\alpha]\approx C[\beta]$, $C[\overline \alpha]\approx C[\overline \beta]$, $C_\infty [\alpha]=C_\infty [\beta]$ and $C_\infty[ \overline \alpha]=C[\overline \beta]$, this finishes the proof. 
\end{proof} 
The following proposition gives the opposite implication:
\begin{prop}\label{finitelyOptimalCm}
Let $\gamma\in \Pi_N(\mu_1,\ldots,\mu_N)$ be a finitely optimal transport plan. Then $\gamma$ is $c$-cyclically monotone or infinitely cyclically monotone, depending on the underlying transport problem.
\end{prop}
 
\begin{proof}First, we present the proof in the case of the integral optimal transport problem. 

Assume on the contrary that there exist points $\{(x_1^i,\ldots,x_N^i)\}_{i=1}^k\subset {\rm supp}(\gamma)$ and $N-1$ permutations $\sigma_2,\ldots,\sigma_{N}$ of the set $\{1,\ldots, k\}$ such that
\begin{equation}\label{violateCm}
\sum_{i=1}^kc(x_1^i,x_2^{\sigma_2(i)},\ldots, x_N^{\sigma_N(i)})<\sum_{i=1}^kc(x_1^i,x_2^i,\ldots , x_N^i).
\end{equation}
Consider the finite submeasure $\alpha$ of $\gamma$, defined by 
\[\alpha=\frac1k\sum_{i=1}^k\delta_{(x_1^i,x_2^i,\ldots,x_N^i)},\]
and the measure 
\[\alpha'=\frac1k\sum_{i=1}^k\delta_{(x_1^i,x_2^{\sigma_2(i)},\ldots,x_N^{\sigma_N(i)})}.\]
Since $\sigma_2,\ldots,\sigma_N$ are bijections, the measures $\alpha$ and $\alpha'$ have the same marginals. However, by inequality (\ref{violateCm}) we have $C[\alpha']<C[\alpha]$, contradicting the finite optimality of $\gamma$. 

The case of the $L^\infty$ optimal transport is analogous. The only difference is that inequality (\ref{violateCm})
takes the form
\[\max_{1\le i\le k}c(x_1^i,x_2^{\sigma_2(i)},\ldots, x_N^{\sigma_N(i)})<\max_{1\le i\le k}c(x_1^i,x_2^i,\ldots , x_N^i).\]
\end{proof}

\section*{Acknowledgement}

The research of the first author is part of the project \emph{Metodologie innovative per l'analisi di dati a struttura complessa} financed by the \emph{Fondazione Cassa di Risparmio di Firenze}.

The second author acknowledges the support of GNAMPA-INDAM and of the PRIN (Progetto di ricerca di rilevante interesse nazionale) 2022J4FYNJ,  \emph{Variational methods for stationary and evolution problems with singularities and interfaces}.

A visit in Firenze of the third author was partially financed with {\it ``Fondi di ricerca di ateneo, ex 60 $\%$''}  of the  University of Firenze.

\end{document}